\documentclass[11pt]{amsart}
\usepackage[T1]{fontenc}

\usepackage{lmodern, amsfonts,amsmath,amstext,amsbsy,amssymb,
amsopn,amsthm,upref,eucal,mathptmx}

\usepackage{mathrsfs}

\RequirePackage{xcolor} 
\definecolor{halfgray}
{gray}{0.55}
\definecolor{webgreen}
{rgb}{0,0.4,0}
\definecolor{webbrown}
{rgb}{.8,0.1,0.1}
\definecolor{red}
{rgb}{1,0,0}
\usepackage{microtype}

\newcommand \R {{ \mathbb R}}
\newcommand \Q {{ \mathbb Q}}
\newcommand \C {{ \mathbb C}}
\newcommand \Z {{ \mathbb Z}}
\newcommand \N {{ \mathbb N}}
\newcommand \T {{ \mathbb T}}

\newcommand \re {{%
\operatorname{Re}
}}
\newcommand \im {{%
\operatorname{Im}
}}
\newcommand{\SL}{%
\operatorname{SL}
}

\newcommand{\GL}{%
\operatorname{GL}
}

\newcommand{\<}{{\langle}} 
\renewcommand{\>}{{\rangle}}

\newtheorem{theorem}{Theorem}[section]
\newtheorem {lemma} [theorem]{Lemma}
\newtheorem {proposition}[theorem]{Proposition}

\newtheorem{remark}[theorem]{Remark}
\newtheorem{claim}[theorem]{Claim}
\newtheorem{definition}[theorem]{Definition}
\newtheorem{conjecture}[theorem]{Conjecture}
\newtheorem{exercise}[theorem]{Exercise}
\newtheorem{problem}[theorem]{Problem}

\title[Effective unique ergodicity and weak mixing of translation flows]%
{Effective unique ergodicity \\ and weak mixing of translation flows}

  \author{Giovanni Forni}

\address{Department  of Mathematics\\
  University of Maryland \\
  College Park, MD USA}
  
\email
    {gforni@math.umd.edu}

\date{\today}

\begin{document}
 
   \maketitle
   \tableofcontents

\section{Introduction}
 
\noindent Translation surfaces and translation flows were introduced in the study of the dynamics of billiards in rational polygons and interval exchange transformations (IET's). Later the renormalization theory of translation flows and IET's provided a fundamental tool for the study of the dynamics of smooth locally Hamiltonian (area-preserving) flows on higher genus surfaces with singularity of saddle type.

In these notes, after reviewing basic definitions and results in the ergodic theory of translation flows, we outline a cohomological approach, based on Hodge theory, to the basic effective ergodic properties of translation flows: polynomial unique ergodicity and polynomial weak mixing.  

We emphasize how this point of view allows for a largely unified treatment to the theory of unique ergodicity and weak mixing of translation flows (and related dynamical systems).  Methods of Hodge theory were introduced 
in the subject in the seminal papers of Kontsevich and Zorich \cite{Ko97}, \cite{KZ97}, developed in the work of the author on effective unique ergodicity starting in \cite{F02} (and in a different direction by M.~M\"oller \cite{M\"o06}),  and perfected in the more recent work of S.~Filip  \cite{Fi16a}, \cite{Fi16b}, \cite{Fi17}.  

The power of the Hodge theoretic approach consist in formulas
and results about the Lyapunov exponents of a linear cocycle (the so-called Kontsevich--Zorich cocycle) which on the one hand encodes  the topological  (homological) behavior of trajectories of translation flows, and on the other hand is related to the tangent cocycle of the  Teichm\"uller flow on the moduli space of Abelian differential. Results on the Lyapunov structure of the Kontsevich--Zorich cocycle are therefore fundamental for both the ergodic theory of translation flows and for the theory of the Teichm\"uller flow and of the related $\SL(2,\R)$ action  (see for instance \cite{ABEM12}, \cite{EM18}). 

Similar methods of Hodge theory (for the twisted cohomology) were first introduced in the author's work on effective weak mixing \cite{F22a}, and have provided a new point of view on earlier work on mixing (in particular a new cohomological proof of a fundamental weak mixing criterion of Veech \cite{Ve84}), as well as formulas and estimates onthe Lyapunov exponents of a ``twisted cohomology cocycle''  which lead to bounds on the speed of weak mixing of typical translation flows.  A similar linear
cocycle was introduced in the work of A.~Bufetov and B.~Solomyak \cite{BS14}, \cite{BS18}, \cite{BS20}, \cite{BS21} on H\"older
property of spectral measures which motivated our own, although the Bufetov-Solomyak cocycle differs as it is not cohomological  
and the analysis of its Lyapunov spectrum is based on a different method (Erd\"os-Kahane argument), not related to Hodge theory.

These notes are thus based mostly on  \cite{F02}, \cite{AtF08}, \cite{FM14} (and also \cite{F06}, \cite{FMZ12})  for unique ergodicity and on~\cite{AvF07}, \cite{F22a} for the weak mixing property of translation flows. The cohomological proof of the Veech criterion for weak mixing was so far unpublished. 

There are many excellent surveys on interval exchange transformations, translation surfaces and flows, Veech surfaces, polygonal billiards, dynamics on moduli space, which can complement these notes, for instance (the list is not exhaustive): \cite{Fi},  \cite{MT02}, \cite{HS06}, \cite{Zo06}, \cite{Yo10}, \cite{FM14}, \cite{Wr15b}.  However, with the exception of \cite{FM14} (which
does not cover weak mixing) none of the above sources covers applications of Hodge theory to the ergodic theory of translation flows.

The notes are an expanded version of lectures given at the CIME school {\it Modern Aspects of Dynamical Systems} in Cetraro, Italy, in August 2021. 

\noindent I am very grateful to the organizers Claudio Bonanno, Alfonso Sorrentino and Corinna Ulcigrai for the opportunity to lecture there.

 \section{Background: definitions and results} 
 
 \subsection{Polygonal Billiards}
 A billiard in an euclidean (planar) polygon $P$ is a dynamical systems defined by the motion of a point mass on the billiard table
 with undergoes ``elastic'' collisions at the boundary edges (according to the reflection law of geometric optics: angle of incidence
 equal angle of reflection).  We refer to the survey of H.~Masur and S.~Tabachnikov \cite{MT02} for an introduction to billiards in
 polygons and translation surfaces and flows.
 
\noindent  We note that the phase space of the polygonal billiard in $P$, restricted to an energy surface, can be take to be $P \times S^1$ since the dynamics on all energy surfaces are isomorphic and the velocity can be taken to be of unit norm, hence an element of $S^1= \{ v\in \R^2 \vert  \Vert v \Vert =1\}$.  We also note that the dynamics of a polygonal billiard is not defined (for 
 all times) for all trajectories that end up in corners, hence it is defined almost everywhere (on the complement of the union of a countable set of lines in the phase space). 
 
 \noindent Polygonal billiards in right triangles are famously related to the motion of two point masses on an interval with elastic collisions between the masses and at the endpoints.
 
 \begin{exercise} Prove that the Hamiltonian system given by  point masses $m_1$ and $m_2$ on the interval $[0,1]$ with elastic
 collisions between the masses and between the masses and the endpoints of the interval is isomorphic to the Hamiltonian system
 given by a right triangle billiard with an angle $\tan^{-1} \sqrt{m_1/m_2}$.
 \end{exercise}

\medskip
\noindent The {\it unfolding construction } of Zemlyakov and Katok \cite{KZ75} uncovers a fundamental dichotomy in the dynamics of polygonal billiards between the so-called {\it rational billiards} and {\it non-rational billiards}. The geometric idea of the construction is to continue a trajectory as a straight line beyond the reflection at an edge by reflecting the table with respect to that edge. 
It replaces therefore the motion with reflections on the billiard table by a straight line flow on a surface endowed with a flat
(indeed a translation) structure. The dichotomy appears since this (minimal) unfolding surface can be a a closed (finite genus) orientable surface (rational case) or a non-closed, infinite surface (non-rational case). Indeed, the unfolding surface is given by glueing  (by translations) along boundary edges of  a number of copies of the given table $P$ equal to the cardinality of the group $G_P$ generated by reflections with respect to lines (through he origin) parallel to edges of $P$. 

\begin{definition} The table $P$ is called rational if and only if the group $G_P$ has finite order. If $P$ is simply connected
the condition that $G_P$ has finite order is equivalent to the condition that all angles of $P$ are rational multiples of $\pi$.
\end{definition} 

\noindent If $P$ is rational, the billiard flow in $P$ is equivalent to the straight line flow on a {\it translation surface}, that is, on
a closed orientable surface $M_P$ with a well-defined horizontal (and vertical) directions (away from finitely many conical points). The straight line flow on a translation surface has a first integral: the angle between the unit tangent vector and the horizontal direction. It follows that the phase space $P\times S^1$  is foliated by invariant surfaces (level surfaces of the prime integral) which are isomorphic to $M_P$ (as translation surfaces).  

\noindent In other terms, the angle of the unit tangent vector with the horizontal on the original table, although not invariant since the trajectory undergoes reflections at the boundary edges, is invariant modulo the action of $G_P$  on $S^1$. If $G_P$ is finite (rational case),
then the angle defines a prime integral with values in the interval $S^1/G_P$, hence the phase space of the billiard flow is foliated
by invariant surfaces. Billiards in rational polygons are a main example of {\it pseudo-integrable}  Hamiltonian systems. 

\smallskip 
\noindent An alternative construction  going back to R.~Fox and R.~Kershner~\cite{FK36} consists in forming the double $S_P$ of the table $P$ (given by two copies of $P$ glued along corresponding edges), which is a flat sphere with finitely many conical singularities with angle which are twice the angle of $P$. 

\medskip
\noindent If all angles are rational multiple of $\pi$, it is possible to consider a branched cover of $S_P$ with suitable branching orders at the cone points to get a higher genus surface endowed with a flat metric with trivial holonomy (that is, with well defined horizontal and vertical directions) and, in particular, with cone singularities of total angles which are {\it integer multiples of }$2\pi$. 

\smallskip
\noindent We conclude this brief introduction to billiards in polygons with the formula for the genus of the unfolding surface of a rational polygonal billiard.

\begin{proposition} (see for instance \cite{MT02})  Let $G_P$ be as above the group generated by all reflections with respect to 
the edges of the rational  polygon $P$. Let $N_P :=  \# G_P/2$ (since $G_P$ is a dihedral group, it has even order). Let 
$\{ \pi m_i  /n_i \vert  i=1, \dots, \sigma\}$ denote the (rational) angles of $P$. Then the genus $g_P$ of $M_P$ is given by the formula:
$$
g_P = 1 + \frac{N_P}{2} ( \sigma -2 - \sum_{i=1}^\sigma \frac{1}{n_i} )\,.
$$
\end{proposition}

 \begin{exercise}  
 Prove the above formula and find all the completely integrable polygons (characterized by $g_P=1$).
 \end{exercise}

\subsection{Moduli space of translation surfaces}
\label{subsec:Mod_spaces}
 The unfolding construction motivated the introduction of the notion of a translation surface.
 
 \begin{definition} The following definitions of a {\bf translation surface} are equivalent :
 \begin{itemize}
 \item A closed orientable surface endowed with a translation structure, that is, endowed with an equivalence class of atlases such that coordinate changes are given by translations of the plane;
 \item A closed orientable surface which can be obtained by glueing planar polygons along boundary edges such that all glueing maps are translations;
 \item  A closed orientable surface endowed with a flat metric with finitely many cone singularities (of total angles integer  multiples 
 of  $2\pi$) and trivial holonomy (that is, such that the holonomy representation of the fundamental group of the complement of the cone points
 is trivial);
 \item  A closed orientable surface endowed with a complex structure, hence a Riemann surface $M$, and a holomorphic differential $h$ on $M$;
 \item A closed orientable surface endowed with a pair $\{X,Y\}$ of transverse vector fields defined on the complement of finitely many points $\{p_1, \dots, p_\sigma\}$ such that $[X,Y]=0$  and such that, for every $i\in \{1, \dots, \sigma\}$ there exists $k_i\in \N$
 and a complex coordinate $z_i$ defined on a neighborhood $U_i$ of $p_i$ such that 
 $$
 X =  \re ( z_i^{-k_i} \frac{ \partial }{\partial z_i} )  \quad \text{ and } \quad  Y=- \im( z_i^{-k_i} \frac{ \partial }{\partial z_i} )  \quad \text{ on } \, U_i\,.
 $$
  \end{itemize}
 The vector fields $X$ and $Y$ are called respectively the horizontal and vertical vector fields of the translation surface.  
\end{definition} 
 
\begin{exercise}  
 Prove the equivalence of the above definitions.
\end{exercise} 
 
\noindent  Linear flows on translation surfaces are called {\it translation flows}:

 \begin{definition} 
 A {\bf translation flow} on a translation surface  $M$ is the flow (defined almost everywhere, outside the union of countable many lines)  generated by a vector field $V$ which is parallel with respect to the flat metric. All translation flows are generated by linear combinations with constant coefficients of the horizontal and vertical vector fields of the translation surface.
 \end{definition}  
 
 \noindent We are interested in the ergodic theory of the ``typical'' translation flow. A precise notion of a typical translation flow can
 be immediately given on every given translation surface since linear flows for normalized generators are naturally parametrized 
 by their angle with the horizontal direction and the set of angles is naturally endowed endowed with the (normalized) Lebesgue measure.  However, most ergodic properties have been proved, at least initially for ``typical'' translation surfaces. This notion
 requires the introduction of finitely many parameters on the space of all translation surfaces and of natural measures on such a
 space, thereby providing motivation for the introduction of a {\it Teichm\"uller space} and a  {\it moduli space} of translation surfaces.
 
\begin{definition}  The moduli space $\mathcal H_g$  of translation surfaces of genus $g\geq 1$ is the quotient of the space of all translation structures on a given smooth closed orientable surface $S$ of genus $g$ under the action of the group $\text{ \rm Diff}^\infty (S)$ of all smooth diffeomorphisms of the surface $S$.  The Teichm\"uller space $\hat{ \mathcal H}_g$ of translation surfaces of genus $g\geq 1$ is the quotient of the space of all translation structures on a given smooth closed orientable surface $S$ of genus $g$ under the action of the subgroup $\text{ \rm Diff}^\infty_0(S)$ of all smooth diffeomorphisms isotopic to the identity.  
Consequently, the moduli space $\mathcal H_g$ can be naturally identified to the quotient $\hat{ \mathcal H}_g/\Gamma_g$ of the
Teichm\"uller space $\hat{ \mathcal H}_g$ under the action of the mapping class group $\Gamma_g :=\text{ \rm Diff}^\infty(S)/\text{ \rm Diff}^\infty_0(S)$. 
\end{definition} 
 
 \noindent Taking into account the equivalent definition of a translation surface as the data of a holomorphic differential on a Riemann surface,
 it can be proved that the Teichm\"uller and moduli spaces of translation surfaces can be respectively identified with finite dimensional vector bundles (of dimension $g$) over  the classical Teichm\"uller and moduli spaces of Riemann surfaces.  This connection allows for methods of Riemann surface theory to be extended and applied to the study of translation surfaces. In particular, the various extensions and refinements of the Deligne-Mumford compactification of the moduli space of Riemann surfaces
 have played a crucial role in the theory of translation surfaces \cite{Ma76}, \cite{MW17}, \cite{BCGGM18} .

 \medskip 
 \noindent The Teichm\"uller and moduli spaces of translation structures are {\it stratified} according to the pattern of multiplicity
 of zeroes of Abelian differentials, or equivalently or total cone angles. For every $ \kappa= (\kappa_1, \dots, \kappa_\sigma)
 \in \N^\sigma$, such that $\sum_{i=1}^\sigma k_i = 2g-2$, the subsets of translation surfaces of genus $g\geq 1$ 
 with a set of cone points 
 of cardinality $\sigma \geq 1$ and multiplicities $\kappa$ is invariant under the action of the diffeomorphism group, hence 
 it descends to a subset  $\hat{\mathcal H}_\kappa \subset  \hat{\mathcal H}_g $  of the Teichm\"uller space and
 to a subset  ${\mathcal H}_\kappa \subset  {\mathcal H}_g$  of the moduli space, called a {\it stratum}. 
 
 \noindent We summarize below several fundamental structures on strata of  Teichm\"uller and moduli spaces of translation structures:
 
 \begin{itemize} 
 \item A function ${\mathcal A}_\kappa: \hat{\mathcal H}_\kappa \to \R^+$ which gives the total area of the flat metric of the translation surface; the function ${\mathcal A}_\kappa$ is invariant under the mapping class group hence is well-defined on 
 the stratum ${\mathcal H}_\kappa$ of the moduli space;
 \item An atlas of affine charts given by the so-called period maps 
 $$
 \hat {\mathcal H}_\kappa \to  H^1(M, \Sigma; \C)\,,
 $$
 with values in the relative cohomology (with complex coefficients) of the surface $M$ relative to a set of $\Sigma=\{p_0, \dots, p_{\sigma-1}\}$ of cone points; the period maps are locally defined on the stratum ${\mathcal H}_\kappa$ of the moduli space
 since the mapping class group acts properly discontinuously on $\hat {\mathcal H}_\kappa$;
 \item A measure class defined on ${\mathcal H} _\kappa$,  and on  $\hat{\mathcal H}_\kappa$, by the pull-back of the Lebesgue measure class on the euclidean space $H^1(M, \Sigma, \C)$  via the local coordinates given by the period map;
 \item A natural action of the Lie group $\GL(2, \R)$ on the stratum $\hat{\mathcal H}_\kappa$, which descends to an action
 on the stratum ${\mathcal H}_\kappa$ since the action of  $\GL(2, \R)$ and of the mapping class group $\Gamma_g$ commute.
\end{itemize}  
 The period maps can be defined  as the (locally defined) map into the relative cohomology with complex coefficients
 $$
 (M, h) \to  [h] \in H^1 (M, \Sigma, \C)
 $$
 or more explicitly as follows: let $\{a_1, b_1, \dots, a_g, b_g, \gamma_1, \dots, \gamma_{\sigma-1}\}$
 denote a standard basis of the relative homology $H_1(M, \Sigma, \Z)$ (recall that $\{a_1, b_1, \dots, a_g, b_g\}$ is a canonical
 basis of the homology and $\{\gamma_1, \dots, \gamma_{\sigma-1}\}$ is  a set of relative cycles joining, for instance, $p_0$
 to $p_1, \dots, p_{\sigma-1}$ respectively). The the period map can be defined as follows:
 $$
 (M, h) \to   (\int_{a_1} h, \int_{b_1} h,  \dots, \int_{a_g} h, \int_{b_g} h, \int_{\gamma_1} h, \dots, \int_{\gamma_{\sigma-1} } h) \in \C^{2g+\sigma-1} \,.
 $$
 The action of the group $\GL(2,\R)$ on $\mathcal H_\kappa$ (hence on $\mathcal H_g$) can be defined as follows:
on translation structures given as the data of a translation atlas, the group $\GL(2, \R)$ acts by post-composition of all charts (which are maps with values in $\R^2$) by a given element of the group;  on translation structures as the data of planar polygons glued along the edges by translations,  the group $\GL(2, \R)$ acts by the natural action on $\GL(2, \R)$ on planar polygons;  on translation structures given as the data $(M,h)$  of a Riemann surface $M$ and a holomorphic differentials $h$, the action
can be described as follows: for all $A\in \GL(2, \R)$, 
$$
A\cdot (M, h) =   (M_A, h_A)    \quad \text{ with }    \begin{pmatrix} \re( h_A) \\   \im(h_A)\end{pmatrix}  =
(A^t)^{-1}   \begin{pmatrix} \re(h)  \\   \im(h) \end{pmatrix} = 
$$
and $M_A$ the unique Riemann surface such that $h_A$ is holomorphic on $M_A$; on translation structures given as the data 
$(X,Y)$ of a pair of infinitesimally commuting vector fields, the action can be described as follows: for all $A=(a_{ij})\in \GL(2, \R)$, 
\begin{equation}
\label{eq:GL_action}
A\cdot (X, Y) =   (X_A, Y_A)    \quad \text{ with }    \begin{pmatrix} X_A \\  Y_A \end{pmatrix}  =
A  \begin{pmatrix} X \\   Y\end{pmatrix}  =   \begin{pmatrix} a_{11} X + a_{12} Y  \\  a_{21} X + a_{22} Y  \end{pmatrix} 
\end{equation}
Since dilations of the translation structure do not affect the dynamical properties of its translation flows, it is natural to restrict
consideration to the subset of translation surfaces of unit total area:
$$
\hat {\mathcal H}^{(1)} _\kappa   = \mathcal A_\kappa^{-1} \{1 \} = \hat {\mathcal H} _\kappa/ \R^+   \quad \text{ and } \quad  {\mathcal H}^{(1)} _\kappa:= \hat {\mathcal H}^{(1)} _\kappa / \Gamma_g\,.
$$

\begin{exercise} Prove that by Riemann bilinear relations the subset $\hat {\mathcal H}^{(1)} _\kappa \subset \hat {\mathcal H} _\kappa  $ is given locally
by a quadratic equation with respect to period coordinates.
\end{exercise} 

\noindent It can be verified that the area function $\mathcal A_\kappa$ is invariant under the action of $\SL(2,\R) <\GL(2, \R)$, hence the action of $\GL(2, \R)$ on $\hat {\mathcal H} _\kappa$ introduced above induces an action of $\SL(2,\R)$ on 
$\hat {\mathcal H}^{(1)} _\kappa$,  that descends to an action of $\SL(2,\R)$ on ${\mathcal H}^{(1)} _\kappa$. 

\smallskip 
\noindent The sub-actions of the above action of $SL(2, \R)$ on the union ${\mathcal H}^{(1)} _g$ of all strata ${\mathcal H}^{(1)} _\kappa$ given by the subgroups $g_\R$ and $h^\pm_\R$, defined (for $t\in \R$) as
\begin{equation}
\label{eq:subgroups}
g_t = \begin{pmatrix} e^t & 0 \\ 0 & e^{-t}   \end{pmatrix}   \quad \text{ and } \quad  h^+_t = \begin{pmatrix} 1& t \\ 0 & 1   \end{pmatrix}
\,, \quad  h^-_t = \begin{pmatrix} 1& 0 \\ t & 1   \end{pmatrix}
\,, 
\end{equation}
are known as the {\it Teichm\"uller geodesic flow} $g_\R$ (introduced by H.~Masur \cite{Ma82}), and the unstable and stable 
{\it Teichm\"uller horocycle flows} $h^\pm_\R$ (introduced in \cite{Ma85}).

\medskip
\noindent A fundamental theorem of H.~Masur \cite{Ma82} and W.~Veech \cite{Ve82} establishes for every stratum the existence
of a probability absolutely continuous measure in the Lebesgue measure class. 

 \begin{theorem} (\cite{Ma82}, \cite{Ve82}) For every stratum $\mathcal H_\kappa$, there exists a 
 unique $\GL(2, \R)$-invariant measure $\mu_\kappa$ on $\mathcal H_\kappa$,  which belongs to  the Lebesgue measure class and projects  to a probability  $\SL(2, \R)$-invariant measure $\mu^{(1)}_\kappa$ on $\mathcal H^{(1)}_\kappa$.  
  \end{theorem} 
 
 \noindent We can now introduce several precise definition of  ``typical''  (in measure sense) for translation surfaces and flows.
 
 \begin{definition} Let $\mu$ denote any $\SL(2, \R)$-invariant probability measure on $\mathcal H^{(1)}_\kappa$. A translation surface in $\mathcal H^{(1)}_\kappa$ is called {\bf $\mu$-typical} if it chosen randomly with respect to $\mu$. It is called {\bf Masur--Veech typical} if it is $\mu$-typical for $\mu=\mu^{(1)}_\kappa$, the Masur--Veech measure. 
 A translation flow is called $\mu$-typical
 if it the horizontal (or vertical) flow of a $\mu$-typical translation surface and Masur--Veech typical if $\mu =\mu^{(1)}_\kappa$ is
 the Masur--Veech measure. Finally, a translation flow is called {\bf directionally typical} for a given translation surface with 
 horizontal/vertical vector fields $(X,Y)$, if it chosen randomly with respect to the Lebesgue measure on the circle $S^1$, among the flows generated by vector fields in the one parameter family
 $$
 \{ V_\theta:=\cos \theta \cdot  X + \sin \theta \cdot Y  \,, \quad \theta \in S^1\}\,.
 $$
 \end{definition} 
 
 \noindent A well-known difficulty in the study of the dynamics of rational polygonal billiards is that they are never Masur--Veech typical (for genus $g\geq 2$),  hence the wealth of results for Masur--Veech typical translation surfaces are not directly relevant for  the dynamics of rational billiards.  
 
 \noindent The strategy to go beyond results that can be proved {\it for all translation surfaces},  for directionally typical translation flows (which of course are relevant for rational billiards), is to consider a natural $\SL(2, \R)$-invariant measure $\mu$ supported  on the $\SL(2, \R)$-orbit closures of a given translation surface, and to prove results for $\mu$-typical translation
 surfaces, hoping that they can be extended to the initial surface. 
 
 This program presents {\it a priori} several very serious difficulties, beginning with the fact that since $\SL(2, \R)$ is not an amenable group, there is no guarantee that an $\SL(2, \R)$-invariant measure supported on any $\SL(2, \R)$-orbit closure even exists. 
The celebrated theorems of Eskin and  Mirzakhani~\cite{EM18}  and Eskin, Mirzakhani  and Mohammadi~\cite{EMM15} made the above strategy possible.

\begin{theorem}  \cite{EM18}, \cite{EMM15} For any $(M,h) \in \mathcal H_\kappa$, its orbits closure $\GL(2,\R)$-orbit closure
$\mathcal M := \overline {\GL(2,\R) (M,h)}$ is locally in period coordinates an affine manifold. In addition, every ergodic probability
measure invariant for the (amenable)  maximal parabolic subgroup $P$ of upper triangular matrices is $\SL(2, \R)$-invariant, hence every $\GL(2,\R)$-orbit closure supports a $\GL(2,\R)$-invariant ergodic measure.  Conversely, every ergodic $\SL(2, \R)$-invariant probability measure  on $\mathcal H_\kappa$ is supported on a codimension $1$ submanifold of an orbit closure (given by
translation surface of a given total area) and belongs to the Lebesgue measure class on it. 
\end{theorem} 

\noindent S.~Filip \cite{Fi16a}, \cite{Fi16b} later refined the above result and proved several algebraic properties (real multiplication and torsion) of orbit closure characterizing (together with a dimensional constraint)  orbit closures, which imply that they are always algebraic varieties defined over $\overline {\Q}$.

\noindent The above mentioned results of Eskin, Mirzakhani and Mohammadi, and Filip have motivated a program of classification of orbit closures. However, from the point of view of questions concerning the dynamical properties of arbitrary translation surfaces,
and of rational polygonal billiards in particular, the above theorem do not establish that the initial translation surface is ''typical''
in its orbit closure (with respect to the associated Lebesgue measure class). 

\noindent In fact, for the purpose of understanding the dynamics of the directionally typical translation flow, for all translation surfaces,
it may be necessary to understand the limits of orbit segments of the circle subgroup or of the unipotent subgroups of $\SL(2, \R)$ 
as they are pushed under the action of the diagonal subgroup. 

\noindent Recent results (see \cite{CW22}, \cite{CSW}) on the dynamics of the unipotent subgroup (Teichm\"uller horocycle flow) have confirmed that it has a complicated dynamical behavior, although they do
not address in general the question of limits of geodesic push-forwards  of horocycle arcs or invariant measures (see \cite{F21} for a partial result).

\subsection{Ergodic properties of translation flows:unique ergodicity} 

\noindent A celebrated theorem proved independently by  H.~Masur  and W.~Veech  states that, in the terminology introduced above:
 
 \begin{theorem}  
 \label{thm:Keane_conj}
 \cite{Ma82} ,  \cite{Ve82} 
 For all strata of translation surfaces, the Masur--Veech typical translation flow is uniquely ergodic. 
 \end{theorem} 
 \noindent Note that, as remarked above, this theorem does not apply to rational billiards. 
 
 \smallskip
\noindent  An equivalent statement for IET's had been conjectured by 
 M.~Keane for IET's and was then known as the {\it Keane conjecture}.  Keane had investigated conditions for the minimality
 of IET's \cite{Ke75}  and had first conjectured that minimality implies unique ergodicity. Counterexamples to this first version  \cite{KN76} prompted him to formulate a revised form of the conjecture. 
 
\noindent  Equivalent results on the minimality of translation flows had been proved independently by Zemlyakov and Katok~\cite{KZ75}, and it appears their paper was not immediately accessible to mathematicians in the West.
   
 \noindent We recall that a dynamical system is uniquely ergodic  if the cone of all probability invariant measures is one-dimensional, a property that, for minimal continuous dynamical systems, is equivalent to the uniform convergence of ergodic averages to the mean, for all continuous functions.  This characterization holds for IET's and translation flows (despite the fact that they are not continuous dynamical systems) for all points with well defined semi-orbits.

The above mentioned Masur--Veech theorem was later extended by Kerckhoff, Masur and Smillie to all translation surfaces. 

\begin{theorem} \cite{KMS86} 
\label{thm:KMS} 
For all translation surfaces, the directionally typical translation flow is uniquely ergodic. In particular, for any rational billiard table
and for Lebesgue almost all directions, the billiard flow is uniquely ergodic on the corresponding invariant surface.
   \end{theorem} 
   
  \noindent An effective version of the above unique ergodicity result was given by Y.~Vorobets \cite{Vo97}, who proved bounds in mean which however are not polynomial (power-law) in time, but have the significant advantage to come with explicit estimates
   for the constants in terms of the genus of the surface. Indeed, the main result of \cite{Vo97} is a condition for the ergodicity (on the $3$-dimensional phase space) of a non-rational polygonal billiard flow in terms of its rational approxximations.
   
   \smallskip
   \noindent  The polynomial unique ergodicity of directionally typical translation flows on all translation surfaces was proved in 
  \cite{AtF08}.
  
  \begin{theorem} 
  \label{thm:eff_erg} 
  For each stratum $\mathcal H^{(1)}_\kappa$ of translation surfaces there exists a constant $\alpha_\kappa>0$
  such that, for all translation surfaces $(M,h) \in \mathcal H^{(1)}_\kappa$ and for the directionally typical translation flow 
  $\phi^{V_\theta}_\R$ on $(M,h)$ there exists a constant $K_h(\theta) >0$ such that the following holds. For all functions 
  $f \in H^1(M)$  (the Sobolev space of functions with square integrable first weak derivative) of zero average, and for all 
  $(x,T) \in M \times\R^+$ such that $x$ has infinite forward orbit, we have
  $$
  \Big \vert  \frac{1}{T}  \int_0^T  f \circ \phi^{V_\theta}_t (x)  dt   \Big\vert   \leq  K_h(\theta) T^{-\alpha_\kappa}\,.
  $$
  \end{theorem} 
   We outline below a cohomological proof of the Masur--Veech unique ergodicity theorem which, thanks to the results of
   J.~Athreya \cite{At06}, can be strengthened to a proof of the above theorem on effective unique ergodicity.

   \subsection{Ergodic properties of translation flows:weak mixing} 
   \noindent We recall the definition of the weak mixing property. It is a standard result of ergodic theory  that weak mixing
   can be defined is several equivalent ways. We consider below the equivalent formulations which are more relevant
   for our purposes. 
   
   \begin{definition} 
   \label{def:wm}
   A flow $\phi_\R$ on a probability space $(M, \mu)$ is \textbf{weakly mixing} if it satisfies any of the
   following equivalent properties: 
   \begin{itemize}
   \item The flow $\phi_\R$ has no non-constant (square-integrable) eigenfunctions, that is, there exists no function
   $u \in L^2(M, \mu)$ of zero average with the property that for some $\lambda \in \R$ 
   $$
   u \circ \phi_t =  e^{2\pi i  \lambda t}  u\,,     \quad \text{ for all } t\in \R\,.
   $$
   \item The spectral measures of the flow $\phi_\R$ for all  functions $f \in L^2(M, \mu)$ of zero average, defined as the Fourier transforms of the self-correlation  functions $\langle f \circ \phi_t, f \rangle$ for $t\in \R$,  are continuous, that is, have no atoms.
  \item The Cesaro averages of correlations with respect to the flow $\phi_\R$ of all pairs of square integrable functions of zero average converge to zero: 
  for all $f, g\in L^2(M,\mu)$ of zero average, we have
  $$
  \frac{1}{T}   \int_0^T  \big \vert \langle  f \circ \phi_t, g  \rangle \big\vert  dt \quad \to \quad 0\;
  $$
\item For all $c\in \R$, the product $\phi_\R\times R^c_\R$ of the flow $\phi_\R$ with the linear flow $R^c_\R$ on $\T$, 
defined as 
$$
R^c_t (y)= y+ ct \,\, \text{ \rm mod. } \Z\,, \quad \text{ for all } (y,t) \in \T\times \R\,,
$$
 is ergodic on the product $(M \times \T, \mu \times \text{\rm Leb}_\T)$.
\end{itemize} 
 \end{definition} 
 
 \begin{exercise} 
 Prove the equivalence of all the properties listed in Definition~\ref{def:wm}.
\end{exercise} 
 
\noindent  We recall that weak mixing is a $G_\delta$ dense properties in the space of measure preserving transformations
 \cite{Ha44}, while mixing is not \cite{Ro48}. For IET's and translation flows, A.~Katok \cite{Ka80} proved that they are never mixing.
 
\smallskip
\noindent The weak mixing property of Masur--Veech typical translation flows on surfaces of higher genus was first proved in~\cite{AvF07}.

\begin{theorem} 
\label{thm:weak_mixing}
For all strata of translation surfaces of genus $g \geq 2$, the Masur--Veech typical translation flow 
is weakly mixing. 
\end{theorem} 

\noindent We note that on the related problem of weak mixing of interval exchange transformations which are not rotations, 
Katok and Stepin \cite{KS67} proved the weak mixing property for IET's on $3$ intervals, and Veech \cite{Ve84} 
generalized the  result to IET's of rotation type on any number of intervals by introducing an important criterion 
for weak mixing which will be explained below. 

\smallskip
\noindent We also note that, as pointed out above, Theorem~\ref{thm:weak_mixing} does not imply any weak mixing property
for rational polygonal billiards. The argument can be generalized to prove the weak mixing property for $\mu$-typical
translation flows for all $\SL(2, \R)$-invariant measures supported on orbit closure of {\it rank} at least $2$ (in the
sense of A.~Wright \cite{Wr15a}). In the rank one case, directionally typical weak mixing was proved by Avila 
and Delecroix \cite{AD16} for translation surfaces, which are not torus covers, on closed $\SL(2, \R)$ orbits 
(non-arithmetic Veech surfaces). This class includes billiards in regular polygons with at least $5$ edges.  A
generalization to all rank one orbifolds has been announced by Aulicino, Avila and Delecroix. 

\smallskip
\noindent Finally, it is well known that weak mixing cannot be directionally typical for translation surfaces which factor over a circle
$\T=\R/\Z$ (that is, that have a completely periodic directional foliation with commensurable cylinders), since non-constant eigenfunctions for the linear flow on $\T$ give by pull-back non-constant eigenfunctions for all ergodic directional flow on the
translation surface.  The presence of exceptional surfaces makes the problem of characterizing surfaces with directionally 
typical weak mixing more challenging. The most natural and simplest  conjecture can be stated as follows:

\begin{conjecture}
For all translation surface which do not factor over the circle (and in particular are not torus covers),  the directionally
typical translation flow is weakly mixing. 
\end{conjecture} 
  
\noindent We now turn to the notion of polynomial weak mixing.   The study of effective weak mixing (in particular of the H\"older property of spectral measures) for substitution systems and translation flows was initiated by A.~Bufetov and B.~Solomyak in a series of papers~\cite{BS14},  \cite{BS18},   \cite{BS20},   \cite{BS21}.

\begin{definition} 
   \label{def:wm_eff}
   A smooth flow $\phi_\R$ on a probability space $(M, \mu)$ is \textbf{polynomially weakly mixing} for functions in a Banach space 
   $W \subset L^2(M, \mu)$  if it satisfies any of the  following (roughly equivalent) properties: 
   \begin{itemize}
   \item The spectral measures of the flow $\phi_\R$ for all  functions $f \in W$ of zero average satsfy a H\"older property:
   there exists $\alpha >0$ such that, for all $\lambda \in \R$ there exists a constant $C(\lambda)>0$ such that the following
   holds. For any $f \in W$ the spectral measure $\sigma_f$ of  $f\in W$ satisfies the bound
   $$
   \sigma_f( \lambda-r, \lambda + r) \leq  C(\lambda) \Vert f  \Vert _W  \cdot  r^\alpha \,, \quad \text{ for all }  r>0 \,;
   $$
   \item The Cesaro averages of correlations with respect to the flow $\phi_\R$ of all pairs of  functions of zero average in $W$ converge to zero polynomially:  there exists $\alpha' >0$ and $C>0$ such that  for all $f, g\in W$ of zero average, we have
  $$
\frac{1}{T}   \int_0^T  \big \vert \langle  f \circ \phi_t, g  \rangle \big\vert  dt \, \leq  \,C \Vert f \Vert_W \Vert g \Vert_W
 \cdot T^{-\alpha'}\,;
  $$
\item The twisted ergodic integrals of functions  in $W$ decay polynomially in mean: there exists $\alpha''>0$, and, for all $\lambda\in \R$,  there exist a constant $C(\lambda) >0$ such that, for all $f\in W$ and all $T>0$, we have
$$
\Vert   \frac{1}{T} \int_0^T  e^{2\pi i  \lambda t} f \circ \phi_t  dt  \Vert_{L^2(M, \mu)}  \, \leq \,   C(\lambda) \Vert f \Vert_W \cdot T^{-\alpha''}\,.
$$
\end{itemize} 
 \end{definition} 
 \begin{exercise} 
 \label{ex:wm_eff}
Discuss the equivalence of all the properties listed in Definition~\ref{def:wm_eff}.
\end{exercise}

\noindent The polynomial weak mixing of Masur--Veech typical translation flows in all higher genus strata was proved in \cite{F22a}
and by a somewhat different approach by A.~Bufetov and B.~Solomyak \cite{BS21}. We will outline below our cohomological approach to (effective) weak mixing based on (twisted) Hodge theory. 

\begin{theorem}
\label{thm:weak_mixing_eff}
For all strata of translation surfaces of genus $g \geq 2$, the Masur--Veech typical translation flow 
is polynomially weakly mixing. 
\end{theorem} 
 \noindent Once more, the argument can be generalized to prove the polynomial weak mixing property for $\mu$-typical
translation flows for all $\SL(2, \R)$-invariant measures supported on orbit closure of {\it rank} at least $2$. 
The rank one case, including the case of  non-arithmetic Veech surfaces is open.

\begin{problem} 
Establish polynomial weak mixing for directionally typical translation flows on non-arithmetic Veech translation surfaces.
\end{problem}
  
 \bigskip
 \noindent  All of the above results are concerned with translation flows and with rational billiard flows 
 We conclude this section with a digression on general polygonal billiards, whose ergodic theory is poorly understood. 
 \begin{conjecture}  
 For almost all (simply connected) polygonal tables the billiard flow is ergodic and weakly mixing (in the $3$-dimensional
 phase space).
 \end{conjecture} 
 \noindent It was proved by S.~Kerckoff, H.~Masur and J.~Smillie \cite{KMS86} that there exists a $G_\delta$ dense set of polygons
 with ergodic billiard flow. Recently J.~Chaika and the author \cite{CF} have  proved that there exists a $G_\delta$ dense set 
 of polygons with weakly mixing billiard flow.  
It is not known whether there exists a polygon with mixing billiard flow, although numerical evidence \cite{CP99}  suggests that it is the case for almost all acute triangles.  However, while (unique) ergodicity and weak mixing are $G_\delta$ properties, dense in the space of measure preserving transformations  \cite{Ha44}, mixing is not \cite{Ro48}. The above-mentioned results on ergodicity and weak mixing of general polygons are proved by fast approximation based on the corresponding properties of rational billiards, an approach that cannot be applied to prove existence of mixing  polygons.
 
 The striking gap in our knowledge of the dynamics of billiards in general polygons when compared to rational billiards is
 related to the discovery (by Rauzy, Masur, Veech) of a renormalization dynamical system for the dynamics of Interval Exchange
 Transformations and Translation flows.

\section{(Effective) unique ergodicity: a cohomological approach} 
 \noindent In this section we review the cohomological approach, based on Hodge theory, to effective unique ergodicity expounded in a series of work including~\cite{F02}, \cite{AtF08}, \cite{FM14}, \cite{Tre14}.  
 
\noindent Recently, it has been rediscovered by C.~McMullen \cite{McM20}, who has often presented it without reference to the much earlier above-mentioned work in his subsequent papers (see for instance \cite{McM23}, page 210).
 
 \smallskip
 \noindent The cohomological approach has its roots in A.~Katok's \cite{Ka73} proof of the finiteness of the cone of invariant measures for quasi-minimal flows on surfaces with saddle-like singularities.  
 
 \subsection{Katok's finiteness theorem}
 
\noindent  In this section,  we outline, following \cite{Ka73}, the proof of the following fundamental result.
 
\begin{theorem} \cite{Ka73}  
\label{thm:Katok_bound}
The cone of invariant measures for any minimal translation flow on a surface $M$ of genus $g \geq 1$  has dimension at most $g$. 
\end{theorem} 
\noindent This result sharpens an earlier finiteness result by Oseledets (who proved for IET's an upper bound equivalent to  $2g + \# \Sigma -1$ for the dimension of the cone of invariant measures of a translation flows on a surface of genus $g$ with singularity set $\Sigma$). 

\noindent In the case $M=\T^2$ ($g=1$), Katok's upper bound already implies that all (linear) minimal flows are uniquely ergodic. 

\medskip
\noindent The first step of Katok's argument consists in introducing the (flux) cohomology class of an invariant probability measure. 
Let $\phi^V_\R$ denote a flow with generator $V$ on $M$ and let $\mu$ an invariant probability measure for $\phi^V_\R$.

\begin{definition}  The flux current $\imath_V \mu$ of the invariant probability measure $\mu$ for the flow $\phi^V_\R$ is the closed current  of dimension~$1$ (and degree~$1$, since $M$ has dimension$2$) defined as follows:
$$
(\imath_V \mu) (\alpha) = \int_M   \imath_V \alpha \, d\mu \,, \quad \text{ for every $1$-form }\alpha \in \Omega^1(M,\C)\,.
 $$
(The symbol $\imath_V$ denote the contraction operator with the vector field $V$ on differential forms, which can be extended to
currents by duality).

\noindent The (flux) cohomology class $F(\mu) \in H^1(M, \R)$ of the invariant probability measure $\mu$ is the (de Rham) cohomology class of the flux form of $\mu$, that is, 
$$
F(\mu) :=  [ \imath_V \mu ]  \in H^1(M, \R) \,.
$$
\end{definition} 
\noindent {\bf Note}: We adopt here the point of view of L.~Schwartz and G.~de Rham that a current of dimension $k\in \N$ is a continuous linear functional on the Fr\'echet space of smooth $k$-forms, and that every closed current of dimension $k$ has a 
well-defined $k$-cohomology class,  defined as the de Rham cohomology class of an appropriate de Rham regularization 
(that is a smooth closed $k$-form).

\smallskip
\noindent In the above definition, it remains to prove that the current $\imath_V \mu$ is indeed closed. This property follows
from the invariance of $\mu$ with respect to $\phi^V_\R$. In fact, we have to prove that the boundary $b(\imath_V \mu)$ or,
equivalently, its differential $d(\imath_V \mu)$ vanish. For every $f \in \Omega^0(M, \C)$, we can compute
$$
\begin{aligned}
d (\imath_V \mu) (f) &=  (\imath_V \mu) (df) =  \int_M   \imath_V df  \, d\mu  \\ &=   \int_M  Vf  \, d\mu 
= \frac{d}{dt}  \left( \int_M  f \circ \phi^V_t   \, d\mu  \right) \Big\vert_{t=0}  = 0\,.
\end{aligned}
$$
We have introduced the language of currents since it is essential in proving effective results. However, following \cite{Ka73}
it is possible to bypass the flux current and directly define the flux cohomology class as an element of the dual $H_1(M,\R)^\ast$.
It is sufficient to define the flux class $F(\mu)$ on a basis of the homology $H_1(M,\Z)$ and then extend it by linearity. As
as basis of the homology, we can choose a canonical basis $\{a_1, b_1, \dots, a_g, b_g\}$ of cycles. It is therefore sufficient to
define the flux through a loop and prove that it only depends on the homology class of the loop. 

\smallskip 
\noindent Indeed, the contraction $\imath_V \mu$ of the invariant measure $\mu$ can be interpreted as a {\it transverse
invariant measure} for the orbit foliation of $\phi^V_\R$ in the following sense.  For every rectifiable arc $I \subset M$, we define
its (transverse) measure $(\imath_V \mu) (I)$ as 
$$
(\imath_V \mu) (I) :=  \lim_{t\to 0}  \frac{1}{t}  \mu \left( \cup_{s\in [0,t]} \phi^V_t (I) \right)
$$
It can be proved that the limit exists and that, in addition,  the one-dimensional transverse measure $\imath_V \mu$ is 
invariant for the orbit foliation: for every pair of transverse arcs $[p, q]$ and $[p', q']$ with endpoints $p,q$ and 
$p', q'$ respectively, such that $p' \in \phi^V_{\R^+}(p)$ and $q' \in \phi^V_{\R^+}(q)$  we have
$$
(\imath_V \mu) ([p,q]) =  (\imath_V \mu) ([p',q'])\,.
$$
The above properties follows from the fact that by Jordan curve theorem the union of the transverse intervals $[p,q]$, $[p',q']$  and of the orbits segments joining $p$ to $p'$ and $q$ to $q'$ bounds a (simply connected) domain $D(p,q,p',q')$ such that, for all
$t>0$, by the invariance of the measure $\mu$ with respect to $\phi^V_\R$, we have 
$$
\begin{aligned}
\mu& \left( \cup_{s\in [0,t]} \phi^V_t ( [p',q']) \setminus \cup_{s\in [0,t] } \phi^V_t ([p,q] )    \right)  \\ & \quad = 
\mu ( \phi^V_t (D(p,q, p',q') ) -  \mu ( D (p,q, p',q') ) =0 \,.
\end{aligned}
$$
Once the notion of a cohomology class of an invariant measure, the argument proceeds by the following steps:
\begin{itemize}
\item if the flow $\phi^V_\R$ is minimal, then the map flux class map $F: \mathcal C_V \to H^1(M, \R)$ is injective from the cone $\mathcal C_V$ of its invariant measures into the cohomology;
\item the image of the map $F: \mathcal C_V \to H^1(M, \R)$ is contained in a Lagrangian subspace for the natural symplectic
structure on $H^1(M, \R)$ (hence it has dimension $\leq g$).
\end{itemize}
The natural symplectic structure on $H^1(M, \R)$ expresses a fundamental symmetry in the structure of deviations of
ergodic averages of translations flows and related dynamical systems.

\smallskip
\noindent We now outline the proof of the above two steps:

\begin{lemma} 
\label{lemma:inj}
If $\phi^V_\R$ is minimal, then $F: {\mathcal C}_V \to H^1(M, \R)$ is injective. 
\end{lemma} 
\begin{proof} Let us assume that for $\mu_1$, $\mu_2 \in {\mathcal C}_V$ we have $[\imath_V(\mu_1)]=[\imath_V(\mu_2)] 
\in H^1(M, \R)$. By the de Rham theorem there exists a current $U$ of dimension $2$ (and degree $0$) such that
$$
dU = \imath_V( \mu_1) - \imath_V(\mu_2)\,.
$$
From the above identity, since $\imath_V^2=0$, it follows immediately that
$$
V U = \imath_V (d U ) = \imath^2_V( \mu_1) - \imath^2_V(\mu_2)=0\,,
$$
hence $U$ is invariant with respect to $\phi^V_\R$. In addition, for  any transverse interval $I \subset M$ and
for any $p, q \in I$ we have, by integrating along the subsegment $[p,q] \subset I$ with enpoints $p$, $q$,
$$
U(p)- U(q) = \int_{[p,q]}  \imath_V( \mu_1-\mu_2) \,.
$$
Since by minimality of the flow,  the current $\imath_V( \mu_1-\mu_2)$ is a continuous transverse measure (it has no atoms), the current $U$ is a continuous function. Since it is invariant $U$ is a constant function, again by the minimality of the flow.  It follows
that $ \imath_V( \mu_1) = \imath_V(\mu_2)$ and, since the contraction operator $\imath_V : \Omega^1(M, \R) \to \Omega^0(M, \R)$ is surjective, it follows that $\mu_1= \mu_2$, thereby concluding the argument.
\end{proof} 

\noindent For every $(p, T) \in M \times \R^+$, let $\gamma^V_T(p)$ denote the (oriented) orbit segment (defined for almost all $p \in M$ for all $T>0$)
$$
\gamma^V_T(p) = \bigcup_{0\leq t \leq T} \{\phi^V_t(p)\} \,,.
$$
For any pair of points $p, q \in M$, let $I(p,q) \subset M$ denote an (oriented) arc of uniformly bounded length (with respect to
a fixed Riemannian metric) joining $p$ to $q$,
union of a subsegment of a transverse arc $I \subset M$ with orbit segments joining $p$ and $q$ to $I$. Let then
$$
\overline{\gamma}^V_T(p) :=  \gamma^V_T(p) \cup I(\phi^V_T(p), p) \,.
$$
\begin{lemma} 
\label{lemma:image}
For any pair of ergodic invariant measures $\mu_1$ and $\mu_2$ for $\phi^V_\R$, we have that
the  intersection of their flux forms vanishes
$$
F (\mu_1) \wedge F(\mu_2) =0 \,,
$$
\end{lemma} 
\begin{proof}
The argument is based on Poincar\'e duality $P: H^1(M, \R) \to H_1(M, \R)$ and Birkhoff ergodic theorem.   
By the ergodic theorem, there exists $p_1, p_2$ in $M$, with distinct orbits, such that the following identity of currents holds : 
$$
\lim_{T\to +\infty} \frac{1} {T}  \overline{\gamma}^V_T(p_1) = \imath_V (\mu_1) \quad \text{ and }  \lim_{T\to +\infty} \frac{1} {T}  \overline{\gamma}^V_T(p_2) = \imath_V (\mu_2) \,.
$$
The above identity concerns currents of dimension $1$, hence it can be checked on smooth $1$-forms. For any ergodic measure $\mu$ and $\mu$-almost all $p\in M$, for every smooth $1$-form $\alpha$, by the ergodic theorem we have 
$$
\lim_{T\to +\infty} \frac{1} {T}  \< \overline{\gamma}^V_T(p), \alpha \> =  \lim_{T\to +\infty} \frac{1} {T} 
\int_0^T \imath_V \alpha (\phi^V_t(p) dt  =  \int_M \imath_V (\alpha) d\mu =  F(\mu) (\alpha) \,.
$$
The cup product $F(\mu_1) \wedge F(\mu_2)$ is therefore given as the limit
$$
F(\mu_1) \wedge F(\mu_2) = \lim_{T\to +\infty} \frac{1}{T^2}  \cdot ( [\overline{\gamma}^V_T(p_1)] \cap [\overline{\gamma}^V_T(p_2)] )\,.
$$
By Poincar\'e duality, the above cup product between cohomology classes is equal to the algebraic intersection number (in homology) between their Poincar\'e duals, which in turn is bounded by the total number of intersections between the loops
$\overline{\gamma}^V_T(p_1)]$ and $\overline{\gamma}^V_T(p_1)]$:
$$
 [\overline{\gamma}^V_T(p_1)] \cap [\overline{\gamma}^V_T(p_2)]  =  P [\overline{\gamma}^V_T(p_1)] \cap P [\overline{\gamma}^V_T(p_2)]   \leq  \#  \left( \overline{\gamma}^V_T(p_1)  \cap \overline{\gamma}^V_T(p_2)  \right)\,.
$$
Finally, we have
$$
\begin{aligned}
  \# &\left(\overline{\gamma}^V_T(p_1)  \cap \overline{\gamma}^V_T(p_2) \right)   =  \# \left({\gamma}^V_T(p_1)  \cap \gamma^V_T(p_2))  +  \# ({\gamma}^V_T(p_1)  \cap I(\phi^V_T(p_2), p_2) \right)  \\ & \quad +  \# \left( I(\phi^V_T(p_1), p_1)  \cap {\gamma}^V_T(p_2) \right) 
  +  \# \left( I(\phi^V_T(p_1), p_1)  \cap I(\phi^V_T(p_2), p_2) \right) \,,
\end{aligned} 
$$
so that since ${\gamma}^V_T(p_1)  \cap \gamma^V_T(p_2) = \emptyset$, as they are distinct orbits of a flow,  and since 
the arcs $ I(\phi^V_T(p_1), p_1)$ have uniformly bounded length,  we conclude that there exists a constant $C>1$ 
such that, for all $T>1$
$$
  \# \left(\overline{\gamma}^V_T(p_1)  \cap \overline{\gamma}^V_T(p_2) \right)  \leq  C  T  \,,
$$
hence 
$$
\begin{aligned}
F(\mu_1) \wedge F(\mu_2) &= \lim_{T\to +\infty} \frac{1}{T^2}  \cdot ( [\overline{\gamma}^V_T(p_1)] \cap [\overline{\gamma}^V_T(p_2)] )
\\ &\leq \lim_{T\to +\infty} \frac{1}{T^2}  \cdot ( \# \left(\overline{\gamma}^V_T(p_1)] \cap [\overline{\gamma}^V_T(p_2)\right) ) =0 \,.
\end{aligned}
$$
\end{proof} 

\begin{proof}[Theorem \ref{thm:Katok_bound}]  By Lemma~\ref{lemma:inj}, the map $F: \mathcal C_V \to H^1(M,\R)$
is injective and, since all  probability invariant measures are convex combinations of ergodic measures, by Lemma \ref{lemma:image} it has image contained in a Lagrangian subspace of the symplectic space $H^1(M, \R)$, endowed with the cup product structure. Since all Lagrangian susbspaces have dimension equal to half the dimension of the space, and $H^1(M, \R)$ 
has dimension~$2g$, the bound on the dimension of $\mathcal C_V$ follows. 
\end{proof} 

\subsection{Self-similar translation flows} 
\label{subsec:selfsim}

A self-similar translation flow is a translation flow rescaled by a  (linear) {\it pseudo-Anosov} map with according to the definition. below. Pseudo-Anosov maps were introduce by W.~Thurston in his work on diffeomorphisms of surfaces.  Those which stabilize
translation flows are of a special type, since their invariant foliations are orientable. In this section we outline a renormalization approach to the effective unique ergodicity of self-similar translation flows. Since the (effective, polynomial) unique ergodicity  of 
the unstable foliation of a hyperbolic (Anosov) diffeomorphism is closely related to the (effective, exponential) mixing of the diffeomorphism,  the argument outlined below is the core of a cohomological approach to exponential mixing for a class of 
pseudo-Anosov diffeomorphisms (see \cite{F22b}).

\begin{definition}
\label{def:self_similar}
 A translation flow $\phi^X_\R$  is self-similar if there exists a translation structure $(M,h)$ with horizontal and
vertical vector fields $(X,Y)$, a homeomorphism $\psi :M \to M$  and a real number $\lambda >1$ such that 
$$
\psi_\ast (X)  =  \lambda X  \quad \text{ and }  \quad   \psi_\ast (Y)  =  \lambda^{-1}  Y\,.
$$
The real number $\lambda$ is called  the dilation factor of the pseudo-Anosov map $\psi$.

\end{definition} 
We note that  the homeomorphism $\psi:M \to M$ which appears in the above definition is a pseudo-Anosov map with orientable
invariant foliations given by the orbit foliations of the translation flows $\phi^X_\R$ and $\phi^{Y}_\R$ and it is a diffeomorphism on the complement of the set $\Sigma_h= \{h=0\}$ of conical singularities of the translation surface.  
\smallskip 
Let $\psi^* : H^1(M, \R) \to H^1(M, \R)$ denote the linear  homomorphism induced by the pseudo-Anosov map $\psi$ on the first
cohomology of the surface. It follows from the definition that
$$
\psi^* ( \re(h) ) = \lambda^{-1} \re(h)  \quad \text{ and }  \quad  \psi^* ( \im(h) ) = \lambda  \im(h) \,.
$$
Since $\re(h)$ and  $\im(h)$ are closed $1$-form, by the de Rham theorem they have well-defined cohomology classes 
$[\re(h)]$ and $[\im(h)]\in H^1(M, \R)$. It follows that the spectrum of the linear
 map $\psi^*$  contains $\{\lambda, \lambda^{-1} \}$. Since $\lambda >1$ equals the maximal dilation of $\psi$ at all points in
 $M\setminus \Sigma_h$, it follows that 
 $$
  \lambda =  \rho (\psi)\,,  \quad   \text{the spectral radius of the linear map } \, \psi^*\,.
 $$
It can proved,  as a consequence of the Perron-Frobenius theorem, that $\lambda$ and $\lambda^{-1}$ are simple eigenvalues
of the linear map $\psi^*$. Another proof, based on Hodge theory, which holds in greater generality, is presented below
 (see Exercise \ref{exercise:simple_1}).

\smallskip
\noindent We now explain how to derive a proof of unique ergodicity of self-similar translation flows based on the ideas
of the previous section.

\begin{theorem} 
\label{thm:UE}
 All self-similar translation flows are uniquely ergodic.
\end{theorem} 
\begin{proof}  We establish first that all  self-similar translation flows are minimal.  Let $(X, Y)$ and $f$ be as in the
above Definition \ref{def:self_similar}. If the translation flow $\phi^X_\R$  is not minimal, then it has a saddle connection $\gamma$, 
as  the surface can be decomposed as the union of finitely many cylindrical component (foliated by periodic orbits) and minimal components. In particular, there exists a saddle connection  of minimal length. However, by the action of $\psi^{-1}$, since $\psi$ and $\psi^{-1}$ stabilize the orbit foliation of $\phi^X_\R$ and $\psi^{-1}$ scales the time variable along the orbits by a factor $\lambda^{-1}$,  we can find a sequence of saddle connections of $\phi^X_\R$ with length converging to zero, which is a contradiction.

\smallskip
Unique ergodicity then follows from the injectivity Lemma~\ref{lemma:inj} and from the following claim:

\begin{claim} \label{claim:UE}  For every ergodic  probability measure $\mu$  for the flow $\phi^X_\R$ on $M$, we have 
$$
\psi^* ( F(\mu) ) =  \lambda  F(\mu) \,.
$$
\end{claim}
\noindent By the above claim, since $\lambda>1$ is a simple eigenvalue and $\psi^*[ \im(h) ] =\lambda [ \im(h) ]$, it follows that, for every ergodic measure $\mu$ of $\phi^X_\R$ on $M$, the flux cohomology class $F (\mu) \in \R [ \im (h) ] \subset H^1(M, \R)$, and, by the injectivity of the map $F: \mathcal C_X \to H^1(M, \R)$, the cone $\mathcal C_X$ of invariant measures for the flow $\phi^X_\R$ is one-dimensional, hence $\phi^X_\R$ is uniquely ergodic.

\begin{proof}[Proof of Claim \ref{claim:UE}] The Claim follows from Birkhoff ergodic theorem. Let $\mu$ be any ergodic probability measure. Let us denote $\gamma^X_T(p)$ the orbit of $\phi^V_\R$  starting at $p$ and ending at $\phi^X_T(p)$, as in the previous section. By the ergodic theorem for $\mu$-almost all $p\in M$ we have (as currents of dimension $1$) almost
$$
F (\mu) = \lim_{T \to +\infty}  \frac{1} {T}   \gamma^X_T(p)  =  \lim_{T \to +\infty}  \frac{1} {T}   \gamma^X_T(\psi(p))    \,,
$$
hence, by taking into account that  $f_\ast (X) = \lambda X$,  we can compute 
$$
\begin{aligned}
\psi_\ast (F(\mu) ) &=  \lim_{T \to +\infty}  \frac{1} {T}   \psi_\ast ( \gamma^X_T(p) ) =  \lim_{T \to +\infty}  \frac{1} {T}  \gamma^X_{\lambda T}(\psi(p) )  \\ &=  \lambda \lim_{T \to +\infty}  \frac{1} {\lambda T}  \gamma^X_{\lambda T} (\psi(p)) =  \lambda F(\mu)  \,.
\end{aligned}
$$
\end{proof}
\noindent The proof of Claim~\ref{claim:UE} completes the proof of Theorem~\ref{thm:UE}.
\end{proof}

The above argument already exhibits the structure of the cohomological argument for unique ergodicity, which derives it
from the simplicity of the top eigenvalue of a linear dynamical system (in this case the action of the pseudo-Anosov map on
cohomology). 

\smallskip
\noindent We now turn to the problem of proving  effective unique ergodicity for self-similar translation flows. The key point
is that, not only $\psi^\ast:H^1(M, \R) \to H^1(M, \R)$ has simple top eigenvalue, but it also has a {\it spectral gap}, that is, the
rest of the spectrum is contained within a disk of radius strictly less than the spectral radius (in the case of homomorphisms of
finite dimensional vector spaces, simplicity implies the existence of a spectral gap).

\begin{theorem} 
\label{thm:eff_UE_selfsim}
There exist $C>0$ and  $\alpha >0$ such that, for all functions $f \in H^1(M)$ (the Sobolev space of square integrable functions with square integrable weak first derivative),  of zero everage, for all $(p,T) \in M \times \R^+$ (such that $p$  has infinite forward orbit) we have
$$
\left \vert \int_0^T  f \circ \phi^X_t (p) dt   \right \vert  \leq  C \Vert  f \Vert_{H^1(M)}  \cdot T^{1-\alpha}\,.
$$
\end{theorem} 
\begin{proof}
{\it The first step} in the argument is to write ergodic integrals in terms of the one-dimensional currents given by orbit segments:
since by definition $\imath_X (\re (h))=1$,  we have, for all  $(p, \R^+)$ and for all $f \in H^1(M)$,
\begin{equation}
\label{eq:erg_int}
 \int_0^T  f \circ \phi^X_t (p) dt  =  \<\gamma^X_T(p),  f \re(h)  \>
\end{equation}
It is therefore enough to prove bounds on the currents $\gamma^X_T(p)$ with respect to the dual Sobolev norm on the
space $W^{-1} (M,h)$ of currents (defined as the dual Banach space of the space $W^1(M,h)$ of $1$-form with coefficients
with respect to $\{\re (h), \im (h)\}$  in the Sobolev space $H^1(M)$. 
We note that, by the {\it Sobolev trace theorem}, the integration currents along smooth $1$-dimensional submanifolds of a surface belong to the dual Sobolev space $W^{-1} (M,h)$. 

\smallskip
\noindent {\it The second step} is to generate orbit segments, and the associate currents, for large $T>0$ by iterated applications
of the pseudo-Anosov map $\psi: M \to M$ on bounded orbit segments. In fact, we have  the following identity. 

\smallskip 
\noindent For any $T>1$, let $n_T = [\log T/ \log \lambda ]$ and let $\tau = T/ \lambda^{n_T} \in [1, \lambda)$. We then have  
\begin{equation}
\label{eq:psi_action}
\gamma^X_T (p)  =    \psi^{n_T}_*   (  \gamma^X_{\tau} ( \psi^{- n_T}(p)   ) )  \,.
\end{equation}
Bounds (and even asymptotics) for the currents $\gamma^X_T(p)$ as a function for $T>1$ can then in principle be derived
from knowledge of the spectrum of the operator $\psi_* : W^{-1} (M,h) \to W^{-1}(M,h)$.  However, this approach does not work since the space $W^{-1} (M,h)$ is too large and bounds on its spectral radius are too weak to derive interesting bounds on ergodic
integrals. 

\smallskip
\noindent {\it The third step}  is based on the remark that although the currents $\gamma^X_T(p)$ are never closed (for minimal
translation flows), they are always at uniformly bounded distance from a closed $1$-current, given by the loop
$$
\overline{\gamma}^X_T (p) = \gamma^X_T(p) \cup I (\phi^X_T(p), p) \,.
$$
In fact,  the integration currents along rectifiable loops are closed and there exists a constant such that
$$
\Vert {\gamma}^X_T (p)  -   \overline{\gamma}^X_T (p)  \Vert_{W^{-1} (M,h)} = \Vert I (\phi^X_T(p), p)    \Vert_{W^{-1} (M,h)} 
 \leq C \text{ \rm diam}(M,h) \,.
$$
 Let $K^{-1}(M,h) \subset W^{-1} (M,h)$ denote the closed subspace defined as
$$
K^{-1}(M,h) = \{ \gamma \in W^{-1} (M,h)  \vert \<\gamma,  \im (h)\>=0 \}\,.
$$
We note that, since $\psi^\ast ( \im (h)) =\lambda  \im (h)$, the subspace $K^{-1}(M,h)$ is 
$\psi_\ast$-invariant. 

\smallskip
\noindent Let then $\mathcal Z^{-1}(M,h) \subset W^{-1} (M,h)$ denote the subspace of closed $1$-currents. 
 For all $k \in \N$ we write 
\begin{equation}
\label{eq:dec}
\psi^k _*  \big(  \gamma^X_{\tau} ( f^{- n_T}(p) \big) +   (\lambda^k \tau)  \im (h)  =   z_k  +  r_k \,, 
\end{equation}
with $z_k \in \mathcal Z^{-1}(M,h) \cap K^{-1}(M,h)$   and $r_k \in  \mathcal Z^{-1} (M,h)^\perp \cap K^{-1}(M,h)$ such that
\begin{equation}
\label{eq:rem_bound1}
\Vert  r_k  \Vert_{W^{-1} (M,h)} \leq C \text{ \rm diam}(M,h)\,.
\end{equation}
By definition, for all $k\in \N$ we have
$$
 z_{k+1}  +  r_{k+1}  = \psi_*( z_k + r_k) =   \psi_*( z_k) +  \psi_*(r_k) \,,
$$
and by orthogonal projection there exists $r'_k\in  \mathcal Z^{-1}(M,h) \cap K^{-1}(M,h) $ such that
\begin{equation}
\label{eq:rec}
z_{k+1}  = \psi_*( z_k) + r'_k\,.
\end{equation}
Since the current $ r'_k \in \mathcal Z^{-1}(M,h)^\perp \cap K^{-1}(M,h)$ is the orthogonal projection of the current $ \psi_*(r_k)$,
by the bound in formula \eqref{eq:rem_bound1}  there exists a constant $C'>0$ such that
\begin{equation}
\label{eq:rem_bound2}
\Vert  r'_k \Vert_{W^{-1} (M,h)} \leq    \Vert \psi_*(r_k) \Vert_{W^{-1} (M,h)}  \leq   C'   \text{ \rm diam}(M,h)  \,.
\end{equation}
Finally let $\rho_\psi \in  [1, \lambda)$ denote the spectral radius of $\psi_\ast$ on $\mathcal Z^{-1} (M,h)$. By the recursive formula \eqref{eq:rec},
we have, for all $k\in \N$, 
$$
\Vert z_{k+1} \Vert_{W^{-1} (M,h)}    \leq    \rho_\psi  \Vert  z_k \Vert_{W^{-1} (M,h)}  + C'  \text{ \rm diam}(M,h) \,.
$$
which by comparison leads to the estimate
$$
\Vert z_k \Vert_{W^{-1} (M,h)} \leq   \rho_\psi^k \left( \Vert z_0 \Vert_{W^{-1} (M,h)}  + C' \text{ \rm diam}(M,h) \sum_{j=0}^{k-1} 
\rho_\psi^{-(j+1)} \right) \,.
$$
For $k = n_T$, by formulas \eqref{eq:psi_action} and \eqref{eq:dec}, we finally conclude that there exists a constant $C''>0$
such that, for all $(p, T)\in M\times \R^+$\,
$$
\begin{aligned}
&\Vert \gamma^X_T(p) +T   \im (h) \Vert_{W^{-1} (M,h)} = \Vert z_{n_T} + r_{n_T} \Vert_{W^{-1} (M,h)}  \\ 
& \qquad\leq  \rho_\psi^{n_T} \left( \Vert z_0 \Vert_{W^{-1} (M,h)}  + C' \text{ \rm diam}(M,h) n_T  \right) \\
 &\qquad\leq   C''  T^{\frac{\log \rho_\psi}{\log \lambda}}  \left(\lambda  + \text{ \rm diam}(M,h) \frac{\log T}{\log \lambda}  \right)
\end{aligned} 
$$
which completes the argument (since $1- \log \rho_\psi/ \log \lambda >0$)  up to the {\it fourth and last step}: the proof of the following claim on the spectrum of $\psi_*$ on the space $\mathcal Z^{-1} (M,h)$ of closed $1$-currents.

\begin{claim} 
\label{claim:spectrum_closed}
The spectrum of $\psi_*$ on the space $\mathcal Z^{-1} (M,h)$ of closed $1$-currents outside of the unit circle
equals the spectrum of the finite dimensional homomorphism $\psi^\ast$ on $H^1(M, \R)$.  In fact, the restriction of $\psi_*$ 
to the space $\mathcal E^{-1} (M,h)$ of exact $1$-currents is isometric, hence its spectrum is contained in the unit circle.
\end{claim}
\begin{proof}[Proof of Claim \ref{claim:spectrum_closed}]
By the de Rham theorem (for currents), there is an isomorphism
$$
H^1(M,\R)  \approx   \mathcal Z^{-1} (M,h) /   \mathcal E^{-1} (M,h) \,.
$$
Since $\mathcal E^{-1} (M,h)$ is a closed $\psi_*$-invariant space, it is enough to prove that $\psi_* \vert \mathcal E^{-1} (M,h)$
is isometric with respect to an equivalent norm. By definition for any exact current $\gamma$ of dimension $1$ there exists
a unique current $U_\gamma$ of dimension $2$ such that  $\<U_\gamma, \text{ \rm area}_h \>=0$ and 
$$
\gamma = d U_\gamma \,.
$$
Since the exterior derivative complex is elliptic with a gain of one derivative, the current $U_\gamma$ is given by integration
against a square integrable function. 

\smallskip
\noindent By the de Rham theorem  the exterior derivative operator $d:\Omega^1(M) \to \Omega^2(M)$ is onto the closed subspace $\Omega_0^2(M)$ of $2$-forms with vanishing total integral. For every smooth 
$1$-form $\alpha$  we can define 
$$
\<U_\gamma, d \alpha \>:= \<\gamma, \alpha\>\,,
$$
and the definition is well-posed since whenever $d\alpha = d\beta$, then $d(\alpha-\beta)=0$ and since $\gamma$ is exact,
by definition
$$
\<\gamma, \alpha-\beta \> =0\,.
$$
Thus $U_\gamma$ is well-defined on the subspace $\Omega_0^2(M)$ and can be extended to a bounded linear functional
on $\Omega^2(M)$  by the condition $\<U_\gamma, \text{ \rm area}_h \> =0$.  

\noindent In addition, by the open mapping theorem, there exists a constant $C>0$ such that, for any smooth function $f$ of zero average on $M$, there exists a smooth $1$ form $\alpha_f$ such that $d \alpha_f = f  \text{ \rm area}_h$,  such that
$$
 \Vert \alpha_f \Vert_{W^1(M,h)}  \leq   C \Vert  f  \Vert_{L^2(M,h)} \,.
$$
It follows that 
$$
\vert \<U_\gamma, f \> \vert  = \vert  \<\gamma, \alpha_f \> \vert \leq  \Vert \gamma \Vert_{W^{-1} (M,h)} \Vert \alpha_f \Vert_{W^{-1} (M,h)}  \leq   C \Vert \gamma \Vert_{W^{-1} (M,h)}   \Vert  f  \Vert_{L^2(M,h)}\,,
$$
which implies that $U_\gamma \in L^2(M,h)$ and that
$$
\Vert  U_\gamma \Vert_{L^2(M,h)} \leq   C \Vert \gamma \Vert_{W^{-1} (M,h)} \,.
$$
Conversely, for all smooth $1$-form $\alpha \in \Omega^1(M)$, we have
$$
\vert  \<\gamma, \alpha \> \vert  = \vert  \<U_\gamma, d\alpha \> \vert  \leq \Vert U_\gamma \Vert_{L^2(M,h)}
 \Vert \alpha \Vert_{W^1(M,h)}\,,
$$
which implies the converse estimate
$$
 \Vert \gamma \Vert_{W^{-1} (M,h)} \leq \Vert  U_\gamma \Vert_{L^2(M,h)}\,.
$$
We have thus proved that the norm 
$$
\Vert\!\vert  \gamma  \vert \!  \Vert\ :=  \Vert  U_\gamma \Vert_{L^2(M,h)} \,, \quad \text{ for all } \gamma \in \mathcal E^{-1} (M,h)\,,
$$
is an equivalent norm on the Banach space $\mathcal E^{-1} (M,h)$ (endowed with the norm induced from $W^{-1} (M,h)$).

\noindent Finally, since the action of a diffeomorphism commutes  with the exterior derivative and the pseudo-Anosov map
is area-preserving on $(M,h)$, hence unitary on  $L^2(M,h)$, we have
$$
\Vert\!\vert  \psi_*(\gamma)  \vert \!  \Vert =  \Vert  U_{\psi_*(\gamma)} \Vert_{L^2(M,h)} =   
 \Vert  \psi^*(U_\gamma) \Vert_{L^2(M,h)}        =          \Vert\!\vert  \gamma  \vert \!  \Vert \,,
$$
which completes the proof that $ \psi_*$ is isometric on $\mathcal E^{-1} (M,h)$ with respect to an equivalent norm.
\end{proof}

\noindent The proof of Claim \ref{claim:spectrum_closed}  completes the proof of Theorem~\ref{thm:eff_UE_selfsim}.
\end{proof}

\subsection{The Kontsevich--Zorich cocycle}

The typical translation flow is not self-similar, and in fact the results of the previous apply only to a countable set of
translations flows (those which are self-similar). To extend the results to typical translation flow, it is necessary to 
introduce a renormalization cocycle over the Teichm\"uller flow (or a related dynamical system such as the Rauzy--Veech 
induction and the Veech ``zippered rectangles '' flow \cite{Ve82}) on a cohomology vector bundle  over the moduli space 
of translation surfaces, which generalizes the action of a pseudo-Anosov map on the cohomology vector space. 

\begin{definition} 
\label{def:Hodge_bundles} 
\cite{Ko97}, \cite{KZ97} The {\bf complex Hodge bundle}  $H^1_\kappa (M, \C)$ over a stratum $\mathcal H^{(1)}_\kappa$ of the moduli space of translation surfaces is the quotient (orbifold) vector bundle $H^1_\kappa (M, \R)$  of the product bundle $\hat{\mathcal H}^{(1)}_\kappa  \times   H^1(M,\R)$ under the product action of the mapping class group $\Gamma_g$  by pull-back:
$$
 H^1_\kappa (M, \C) :=   \left( \hat{\mathcal H}^{(1)}_\kappa  \times   H^1(M,\C) \right) /  \Gamma_g 
$$ 
The {\bf real  Hodge bundle}  $H^1_\kappa (M, \R)$ is the real part of the complex
Hodge bundle $H^1_\kappa (M, \C)$, that is, the quotient of the trivial bundle $\hat{\mathcal H}^{(1)}_\kappa  \times   H^1(M,\R)$. 

\smallskip
\noindent The {\bf Kontsevich--Zorich cocycle} $g^{\rm KZ}_\R$ on $H^1_\kappa (M, \C)$   is the projection to the complex Hodge bundle under the action of the mapping class group $\Gamma_g$  of the product cocycle (over the Teichm\"uller geodesic flow $g_\R$)
$$
g_\R \times  \text{\rm Id} :   \hat{\mathcal H}^{(1)}_\kappa  \times   H^1(M,\C)  \to   \hat{\mathcal H}^{(1)}_\kappa  \times   H^1(M,\C) \,.
$$
The Kontsevich--Zorich cocycle on  the complex Hodge bundle $H^1_\kappa (M, \C)$  has a well-defined restriction to  the real Hodge bundle $H^1_\kappa (M, \R)$. 
\end{definition} 
\begin{remark}  The trivial bundle $\hat{\mathcal H}^{(1)}_\kappa  \times   H^1(M,\C)$ is well-defined over the Teichm\"uller space 
of Riemann surfaces, despite the fact that the Teichm\"uller space is an equivalence class of surfaces with respect to the action of the group $\text{\rm Diff}^+_0(M)$ of diffeomorphisms isotopic to the identity, since the action of $\text{\rm Diff}^+_0(M)$ is trivial. 

\noindent The bundle $H^1_\kappa (M, \C)$ is therefore the pull-back of a bundle well-defined over the moduli space of Riemann surfaces (forgetting the translation surface structure). The trivial connection on $\hat{\mathcal H}^{(1)}_\kappa  \times   H^1(M,\C)$ projects to a connection on $H^1_\kappa (M, \R)$ called the {\it Gauss-Manin connection}. 

\noindent  The Kontsevich--Zorich cocycle is given by
the parallel transport of cohomology classes along orbits of the Teichm\"uller flow with respect to the Gauss-Manin connection.
\end{remark} 

\begin{exercise} Let $(M, h) \in \mathcal H_\kappa^{(1)}$ such that there exists a pseudo-Anosov map  $\psi: (M, h) \to (M,h)$ (as in Definition \ref{def:self_similar}) of dilation factor $\lambda >1$.  Derive from from the definition of the Teichm\"uller flow $g_\R$ (see formulas \eqref{eq:GL_action} and \eqref{eq:subgroups} in section\ref{subsec:Mod_spaces}) that $(M,h)$ is a periodic point of $g_\R$ on $\mathcal H_\kappa^{(1)}$ of period 
$T=\log \lambda$.   Prove that all periodic orbits of the Teichm\"uller flow on  $\mathcal H_\kappa^{(1)}$ are of this type. Prove that
the time $T$-map of the Kontsevich--Zorich cocycle is given by the formula
$$
g^{\text {\rm KZ}}_T ( (M,h), c) =   ( (M, h), (\psi^{-1})^* (c) ) \,, \quad \text{ for all }   c \in H^1(M, \R)\,. 
$$
\end{exercise} 
\noindent The Kontsevich--Zorich cocycle generalizes the action of a pseudo-Anosov map on cohomology, and provides a framework to extend the results of section \ref{subsec:selfsim} to typical translation flows. In terms of the action  of elements of the mapping class group on the cohomology, the returns of an orbit of the Teichm\"uller geodesic flow to a fundamental domain of a stratum of the moduli space of translation surfaces encode a sequence of (pseudo-Anosov) maps whose compositions replace the iteration of a single map, as in the self-similar case. 
A natural class of translation flows for which it is possible to prove (effective) unique ergodicity by these methods is therefore that 
of horizontal flows of translation surfaces which are recurrent under the Teichm\"uller flow (with strictly positive frequency). 

\smallskip
\noindent  As the Kontsevich--Zorich cocycle generalizes the action of a single diffeomorphism on the cohomology vector
space, the notion of a {\it Lyapunov spectrum} generalizes the notion of the spectrum of a homomorphism (strictly speaking
the Lyapunov exponents generalize the logarithms of the modulus of the eigenvalues).  Since  the symplectic structure given 
by the cup product on the cohomology $H^1(M, \R)$ is invariant under the action of diffeomorphisms, the Kontsevich-Zorich 
is a {\it symplectic cocycle}, hence its Lyapunov spectrum, with respect to any probability $g_\R$-invariant measure $\mu$ on
a stratum $\mathcal H_\kappa^{(1)}$, is symmetric:
$$
\lambda^\mu_1 \geq \lambda^\mu_2 \geq \dots \geq \lambda_g^\mu \geq - \lambda_g^\mu \geq \dots \geq  \lambda^\mu_2 \geq -  \lambda^\mu_1\,.
$$
The {\it Kontsevich--Zorich conjectures} \cite{Ko97}, \cite{KZ97} stated that for $\mu=\mu^{(1)}_\kappa$ the Masur-Veech
measure on a stratum of the moduli space 
\begin{equation}
\label{eq:KZ_exp}
\lambda^\mu_1=1 > \lambda^\mu_2 > \dots > \lambda_g^\mu >0\,,
\end{equation}
and that these exponents are also the deviation exponents of ergodic averages of {\it Masur--Veech  typical }translation flows 
(as well as Masur--Veech  typical locally Hamiltonian flows with non-degenerate saddle singularities) in the sense that, there exists linear functionals $D_2, \dots, D_g$ on $C^\infty(M)$ (invariant distributions) such that, for all functions $f \in C^\infty(M)$ and for all $\epsilon >0$, 
\begin{equation} 
\label{eq:KZ_asymptotics}
\int_0^T f (\phi^X_t(x) ) dt  = T \int_M f \text{\rm area}_h  +   \sum_{i=2}^g D_i(f)  T^{\lambda_i + o(1)} +  O (T^\epsilon)\,.
\end{equation}
The identity $\lambda^\mu_1=1$ is elementary (as it be explained below), the {\it spectral gap} property 
$\lambda_2^\mu < 1$ can be derived from Veech proof of non-uniform hyperbolicity of the Masur--Veech
measures, and will be proved below (following \cite{F02}) for all probability ergodic measures.  In fact, it follows
from the existence of period coordinates that the {\it non-negative} half of symmetric  Lyapunov spectrum of the tangent cocycle 
$Tg_\R$ of the Teichm\"uller flow are
$$
2 \geq 1+ \lambda^\mu_2 \geq \dots \geq 1+\lambda^\mu_g \geq 1 = \dots =1 \geq 1-\lambda^\mu_g \geq \dots  \geq 1-\lambda^\mu_2 \geq  0   
$$
(with the exponent $1$ appearing with multiplicity at least equal to the number of cone points of surfaces in the stratum),
hence $\lambda^\mu_2<1$ if and only if all the non-trivial Lyapunov exponents of the tangent cocycle $Tg_\R$ are non-zero.

\smallskip
\noindent The non-vanishing
of Kontsevich--Zorich exponents, that is, the strict inequality $\lambda^\mu_g>0$ (for $\mu= \mu^{(1)}_\kappa$),
as well as the asymptotic of ergodic averages \eqref{eq:KZ_asymptotics}  for translations flows, were proved in \cite{F02} 
(see also  \cite{F11}). 

\noindent The {\it simplicity} the the  Kontsevich--Zorich spectrum, that is, all the strict inequalities 
in formula \eqref{eq:KZ_exp}, was proved by  A.~Avila and M.~Viana \cite{AV07}.  

\noindent The Kontsevich--Zorich conjecture 
on the deviation of ergodic averages for locally Hamiltonian flows has been completely proved only recently \cite{FU}, and
generalized to flows with degenerate (canonical) saddle like singularities in \cite{FK}. 

\smallskip
\noindent By the Oseledets theorem, for  almost all translation surfaces $(M,h)$ (with respect to any $g_\R$-ergodic measure $\mu$
on $\mathcal H_\kappa^{(1)}$) there exists a splitting of the cohomology 
$$
H^1(M, \R) = E(\lambda_1) \oplus  E(\lambda_2)  \dots \oplus  E(\lambda_s) \oplus E(-\lambda_s) \oplus \dots \oplus E(-\lambda_2) \oplus E(\lambda_1)
$$
into subspaces corresponding to the distinct non-negative Lyapunov exponents  
$$
\lambda^\mu_1 >\lambda^\mu_2 >\dots > \lambda^\mu_s > -\lambda^\mu_s > \dots > -\lambda^\mu_1\,,
$$ 
 such that for all $i\in \{1, \dots, s\}$ and for all $c \in  E(\pm \lambda_i)$,
$$
\lim_{t\to \pm\infty}  \frac{1}{t}  \log \Vert  g^{\text{\rm KZ}}_t (c)   \Vert  = \pm \lambda_i\,.
$$
In particular, if $\lambda_1=1$ is simple and $\lambda_2<\lambda_1$, then for  $c \in H^1(M, \R)$ such that
$ c \wedge E_1(\lambda_1) =0$ we have
$$
\lim_{t\to \pm\infty}  \frac{1}{t}  \log \Vert  g^{\text{\rm KZ}}_t (c)   \Vert  < \lambda^\mu_1=1\,.
$$
In the following we will prove a similar, but stronger upper bound for all  translation surfaces $(M,h)$ with the property  
that the Teichm\"uller geodesic ray $g_\R(M,h)$ returns with positive frequency to any given compact subset of the
moduli space (of Riemann surfaces). 

\subsection{The Hodge norm and its first variation} 
\label{subsec:Hodge_norm}

A powerful tool to analyze the Kontsevich--Zorich cocyle (hence the tangent cocycle of the geodesic flow) is given by
the Hodge norm on the cohomology bundle (see  for instance \cite{Ko97}, \cite{KZ97}, \cite{F02}, \cite{ABEM12} , 
\cite{EM18}, \cite{McM20}, \cite{Fr22}, \cite{KW22}). 

\begin{definition}   Let $M$ be a Riemann surface. The {\bf Hodge norm on the complex cohomology} $H^1(M, \C)$
is defined as follows. For every $c \in H^1(M, \C)$ and for every complex closed $1$- form  $\alpha$ on $M$ such
such that $c= [\alpha]$ we set
$$
\Vert c  \Vert_{H^1(M, \C)}  := \Vert \alpha \Vert_M = \left( \frac{ \imath}{2}  \int_M  \alpha \wedge \overline {\alpha}\right)^{1/2} \,. 
$$
The {\bf Hodge norm on the real  cohomology} $H^1(M, \R)$ can be defined  via the identification (Hodge representation
theorem for Riemann surfaces) of the real cohomology 
$H^1(M, \R)$ with the subspace of holomorphic $1$-forms $H^{1,0}(M) \subset H^1(M, \C)$. 

\noindent In other terms, for every 
$c \in H^1(M, \R)$  and for every holomorphic $1$- form $\alpha$ on $M$ such such that $ c= \re[\alpha]$ we set
$$
\Vert c  \Vert_{H^1(M, \R)}  := \Vert \alpha \Vert_M = \left( \frac{ \imath}{2}  \int_M  \alpha \wedge \overline {\alpha} \right)^{1/2} \,.
$$
\end{definition} 

\begin{exercise} Prove that the Hodge norm is well-defined, that is, the definition is independent of the $1$-form
representing  the (de Rham) cohomology class, and that it gives an hermitian  norm on $H^1(M, \C)$ and an euclidean
norm on $H^1(M, \R)$.
\end{exercise}

\begin{exercise} Prove that the Hodge norm is equivariant under the action of group $\text{ \rm Diff}^+(S)$  of the
underlying smooth surface $S$, in the sense that,  for evert $\phi  \in \text{ \rm Diff}^+(S)$ and for every $c\in H^1(M, \C)$,
we have 
$$
\Vert  \phi ^\ast (c)     \Vert_{H^1( \phi_\ast(M), \C)}  =  \Vert  c     \Vert_{H^1(M, \C)}  \,.
$$
The Riemann surface $\phi_\ast(M)$  is defined by precomposition of the charts of an atlas of the Riemann surface $M$
with the diffeomorphism $\phi$ on $S$. 

\noindent Prove that, since it is equivariant, the Hodge norm induces a  norm on the  Hodge bundle $H^1_\kappa (M, \C)$,
and on $H^1_\kappa (M, \R)$.
\end{exercise} 

\noindent {\bf Notation}: In the following we will adopt the following notation for  the Hodge norm of a cohomology class $c \in H^1(M, \R)$ at a translation surface  $(M, h) \in \mathcal H^{(1)}_\kappa$:
$$
\Vert  (c, M, h) \Vert  =  \Vert c    \Vert_{H^1(M, \R)} \,.
$$
The first variation of the Hodge norm along a Teichm\"uller geodesic is given by the following formula  (see \cite{F02}, Lemma 2.1' ,
 \cite{FM14}, Theorem 29, or \cite{FMZ12}, Lemma 2.5):
\begin{lemma} \label{lemma:first_var} 
Let $(M, h) \in \mathcal H^{(1)}_\kappa$. For any $c \in H^1(M, \R)$,  let $\alpha  \in H^{1,0}(M)$
denote the unique holomorphic differential such that $c=\re(\alpha)$. We have 
$$
\frac{d}{dt} \Vert g^{ \text{\rm KZ}} _t (M,h, c)\Vert^2  \Big\vert _{t=0} =   2 \re\, B_{(M,h)}(c) :=  2 \re\left( \int_M  \left(\frac{\alpha}{h} \right)^2 \text{\rm area}_h \right)\,.
$$
\end{lemma} 
\begin{proof}  Since the Hodge norm is equivariant under the action of the mapping class group $\Gamma_g$ we
prove the identity for the product cocycle on $\mathcal H^{(1)}_\kappa \times H^1(M, \R)$.

\noindent For all $t\in \R$, let $g_t (M, h) = (M_t, h_t)$ and let $\alpha_t = f_t  h_t \in H^{1,0} (M_t)$ denote the unique 
$1$-form, holomorphic on the Riemann surface $M_t$, such that $c =[\re(\alpha_t)]$.  

By definition of the Hodge norm we have 
$$
\frac{d}{dt} \Vert g^{ \text{\rm KZ}} _t (M,h, c)\Vert^2 =   \frac{d}{dt}  \left(\frac{\imath}{2}\int_M \alpha_t \wedge   \overline{\alpha_t}\right) =
- \re \int_M \alpha_t \wedge \overline{ \frac{d\alpha_t}{dt} }  \,.
$$
Since $[\re(\alpha_t)]$ is constant, it follows that, 
hence the $1$-form 
$$
\re(\frac{d\alpha_t}{dt})  \quad  \text{ is an exact $1$-form.}
$$
Since $\alpha_t$ is an holomorphic form, hence closed,  and $\re(\frac{d\alpha_t}{dt} )$ is exact, we have
$$
\int_M  \alpha_t  \wedge   \overline{ \frac{d\alpha_t}{dt} } = \int_M  \alpha_t  \wedge  \left( \overline{\frac{d\alpha_t}{dt}} -
2\re( \frac{d\alpha_t}{dt} ) \right) = - \int_M \alpha_t \wedge  \frac{d\alpha_t}{dt}  \,.
$$
By definition of the Teichm\"uller flow, we have $h_t = e^{-t}  \re (h) +  \imath e^t \im(h)$, hence
\begin{equation}
\label{eq:h_t_der}
 \frac{d h_t}{dt} = -e^{-t}  \re (h) +  \imath e^t \im(h) = - \overline{h_t} \,,
\end{equation}
hence we have
$$
\frac{d \alpha_t}{dt} = \frac{d f_t}{dt} h_t + f_t \frac{d h_t}{dt}=  \frac{d f_t}{dt} h_t - f_t \overline{h_t}\,,
$$
hence, by taking into account that  $h_t$ is holomorphic,
$$
\int_M \alpha_t \wedge  \frac{d\alpha_t}{dt} = \int_M \alpha_t \wedge  \frac{d\alpha_t}{dt} =
\int_M f_t h_t \wedge (\frac{d f_t}{dt} h_t - f_t \overline{h_t}) = - \int_M f_t^2 \,h_t \wedge \overline{h_t}  \,.
$$
We conclude that, since $h_t \wedge \overline{h_t} =-2 \imath  \text{ \rm area}_h$, for all $t\in \R$, 
$$
\frac{d}{dt} \Vert g^{ \text{\rm KZ}} _t (M,h, c)\Vert^2 = - \im\left(\int_M f_t^2 \,h_t \wedge \overline{h_t} \right)
= 2  \re \left(\int_M f_t^2 \, \text{ \rm area}_h \right)\,,
$$
as stated (since by definition $f_t = \alpha_t/h_t$, for all $t \in \R$).

\end{proof}
\begin{remark} For every $(M,h) \in \mathcal H^{(1)}_\kappa$, the bilinear form 
$$
B_{(M,h)}(\alpha, \beta) = \frac{\imath}{2} \int_M   \frac{\alpha}{h}  \frac{\beta}{h} \, h \wedge \bar h =  
\frac{\imath}{2} \int_M  (\alpha \beta)\, \frac{ \bar h} {h} \,, \quad \text{ for all } \alpha, \beta \in H^{1,0}(M)
$$
can be interpreted as the {\it fundamental form} of the holomorphic connection with respect to the (flat)
Gauss-Manin connection (see~\cite{FMZ12}, section 2.3).
\end{remark} 

\noindent We prove a crucial estimate on the second fundamental form:

\begin{lemma} 
\label{lemma:fund_form}
For every $(M,h) \in  \mathcal H^{(1)}_\kappa$ and for all $c \in H^1(M), \R)$, we have
$$
\vert B_{(M,h)} (c) \vert  \leq  \Vert (M,h, c) \Vert^2 \,.
$$
Let then $\mathcal T (M,h) := \R[\re(h)] \oplus \R [\im(h)]  \subset H^1(M, \R)$ denote the so-called 
{\it tautological plane}   and let $\mathcal T (M,h)^\perp$ denote its symplectic orthogonal:
$$
\mathcal T (M,h)^\perp =\{ c \in H^1(M, \R) \vert   c \wedge [\re(h)] = 
c \wedge [\im(h)] =0\}\,.
$$ 
We have  the following estimate: 
\begin{equation}
\label{eq:Lambda}
\Lambda (M,h) = \max_{ c \in \mathcal T (M,h) \setminus \{0\} }  \frac{ \vert B_{(M,h)} (c) \vert  } { \Vert (M,h, c) \Vert^2   }  \,  < \, 1\,.
\end{equation}
\end{lemma}
\begin{proof} Let $c = \re[\alpha]$ with $\alpha\in H^{1,0}(M)$ and let $f= \alpha /h$. A straightforward estimate
based on the Schwarz inequality gives
$$
\vert B_{(M,h)} (c) \vert =  \big\vert  \int_M  f^2  \text{ \rm area}_h \big\vert = \vert \langle  f , \bar f  \rangle_{L^2(M, \text{ \rm area}_h)}    \leq   \Vert  f \Vert^2_{L^2(M, \text{area}_h)} = \Vert (M,h,c)\Vert^2\,.
$$
The first bound in the statement is therefore proved.

\smallskip
\noindent In addition, by the Cauchy-Schwarz inequality, equality holds if and only if  $\bar f \in \C  f$. Since  the differential 
$\alpha = f h \in H^{1,0}(M)$, so that the function $f$ is meromorphic (and not equal to zero), if equality holds then
$f$ is meromorphic and ant-meromorphic (as $\bar f$ is also meromorphic), that is, if and only if  $f$ is a non-zero constant  function. However, for all cohomology classes $c \in \mathcal T (M,h)^\perp$ $f$ is not a non-zero constant function since
$$
  \int_M f  \text{ \rm area}_h =  \imath  \int_M  \re(f h)  \wedge \bar h   = 0 \,.
$$
hence the above  Cauchy-Schwarz inequality is strict. The proof of the second inequality is therefore complete.

\end{proof} 

\noindent The (tautological) $g^{\text{ \rm KZ}}_\R$-invariant expanding and contracting  line bundles $\mathcal T^\pm$ 
are defined as 
$$
\mathcal T^+ (M,h) = \R  [\re(h)]  \quad \text{ and } \quad  \mathcal T^- (M,h)  = \R  [\im(h) ] 
$$
have well defined Lyapunov exponents equal to $\pm 1$ since
$$
g^{\text{\rm KZ}}_t ([\re(h)] ) =  e^t [\re(g_t(h))]    \quad \text{ and } \quad  g^{\text{\rm KZ}}_t ([\im(h)] ) =  e^{-t} [\im(g_t(h))]   
$$

\begin{exercise} Prove that for any $g_\R$-ergodic invariant measure $\mu$ on $\mathcal H^{(1)}_\kappa$ the top 
Kontsevich--Zorich exponent $\lambda^\mu_1 =1$. 

\end{exercise} 

\noindent The following lemma is the key ``spectral gap''  result for the Kontsevich--Zorich cocycle:

\begin{lemma}  
\label{lemma:spectral_gap}
For all $(M,h) \in \mathcal H^{(1)}_\kappa$, for all $c \in \mathcal T(M,h)^\perp$ and for all $t\in \R$, we have
$$
\Vert  g^{\text{\rm KZ}}_t (M, h, c) \Vert  \leq   \Vert (M, h, c) \Vert    \exp \left( \int_0^t  \Lambda \big( g_s (M, h) \big)  ds \right)\,.
$$
In particular, if the forward Teichm\"uller  orbit $g_\R (M,h)$ visits a compact set $K \subset  \mathcal H^{(1)}_\kappa$ with positive frequency,  in the sense that 
\begin{equation}
\label{eq:f_K}
f_K := \liminf_{t \to + \infty}  \text{ \rm Leb} ( \{  t \geq 0 \vert   g_t(M,h) \in K \} ) > 0\,,
\end{equation}
then, for all $c \in \mathcal T(M,h)^\perp$, 
$$
-1 < \liminf_{t \to \infty}  \frac{1}{t} \log  \Vert  g^{\text{\rm KZ}}_t (M, h, c) \Vert \leq  \limsup_{t \to \infty}  \frac{1}{t} \log  \Vert  g^{\text{\rm KZ}}_t (M, h, c) \Vert  <1 \,.
$$
\end{lemma} 
\begin{proof} From Lemma~\ref{lemma:first_var}, for all  $(M,h) \in \mathcal H^{(1)}_\kappa$, all $c\in H^1(M, \R)$
and for all $t\in \R$,  we have
$$
\frac{d}{dt} \log \Vert  g^{\text{\rm KZ}}_t (M, h, c) \Vert =    
\frac{\re\, B_{g_t(M,h)}(c)} {\Vert  g^{\text{\rm KZ}}_t (M, h, c) \Vert      } \,.
$$
By the above formula and by the definition of the function~$\Lambda$ in formula~\eqref{eq:Lambda}, it follows that for all $(M,h)\in 
\mathcal H^{(1)}_\kappa$ and for all $c \in \mathcal T(M,h)^\perp$, 
$$
\frac{d}{dt} \log \Vert  g^{\text{\rm KZ}}_t (M, h, c) \Vert \leq \Lambda ( g_t (M,h) ) \,,
$$
hence the estimate in the statement follows by integration. 

\smallskip
\noindent It can be proved that the function $\Lambda$ is continuous on $\mathcal H^{(1)}_\kappa$ (as the maximum of
a continuous functions on the Hodge bundle over the unit sphere in the fiber).  Since $\Lambda <1$ everywhere, it follows that for any compact set $K \subset \mathcal H^{(1)}_\kappa$,
$$
 \Lambda_{\max} := \sup_{ (M,h) \in K} \Lambda (M,h)  = \max_{ (M,h) \in K} \Lambda (M,h)  <1 \,.
$$
Under the assumption  that the visit frequency $f_K>0$ (see formula~\eqref{eq:f_K})  we can estimate
$$
\limsup_{t \to \infty}  \frac{1}{t} \int_0^t  \Lambda \big( g_s (M, h) \big)  ds \leq  1- (1- \Lambda_{max}) f_K  <1\,,
$$
from which the estimate in the statement follows immediately.
\end{proof} 

\begin{exercise} 
\label{exercise:simple_1}
Derive from Lemma~\ref{lemma:spectral_gap} a proof  (announced in section  \ref{subsec:selfsim})  that the action on cohomology of  all pseudo-Anosov map with orientable invariant foliations has a simple top eigenvalue. 
\end{exercise} 
In fact, the Hodge theory approach easily gives a more general result:
\begin{exercise} 
\label{exercise:simple_2}
Derive from Lemma~\ref{lemma:spectral_gap} a proof  that  for any ergodic $g_\R$ invariant measure  the top Lyapunov exponent is simple and we have the spectral gap
$$
\lambda^\mu_2 < \lambda^\mu_1=1\,.
$$
\end{exercise} 

\begin{exercise} 
\label{exercise:simple_3}
Derive from Lemma~\ref{lemma:spectral_gap} a proof  that  any ergodic $g_\R$ invariant measure  is non-uniformly hyperbolic,
in the sense that all Lyapunov exponents of the tangent cocycle $Tg_\R$ are non-zero, except the (trivial) zero exponent corresponding to the flow direction.
\end{exercise}

\begin{remark} Lower bounds on Kontsevich--Zorich exponents of $\SL(2,\R)$-invariant measures can be derived from second
variation formulas for he Hodge norm \cite{F02}, in particular from a formula  for the (hyperbolic) Laplacian of the Hodge norm along $\SL(2,\R)$ orbits  (Teichm\"uller disks). However, the examples of the so-called Eierlegende Wollmilchsau~\cite{F06} and Platypus
\cite{FM08}, \cite{FMZ11} made clear that there is no general lower bounds, as the second exponent can be equal to zero. A criterion for the positivity of Kontsevich--Zorich exponents was given in~\cite{F11}, improving upon~\cite{F02}.
\end{remark} 

\subsection{Typical (effective) unique ergodicity} 
In this section we outline the cohomological proof of (effective) unique ergodicity. The arguments generalize those given in
the self-similar case in section \ref{subsec:selfsim}, however we adopt a different, more geometric, perspective, which 
replaces the analysis of the action of the mapping class group on cohomology with control of the flat geometry give by the renormalization
dynamics (Teichm\"uller flow) on the moduli space.  

\smallskip
\noindent We begin with a cohomological proof of a version of a fundamental unique ergodicity criterion 
first proved by H.~Masur (see \cite{Ma82}, Prop. 6.2):

\begin{theorem} (Masur's criterion) 
\label{thm:Masur_criterion}
Assume that the forward Teichm\"uller orbit  $g_{\R^+} (M,h)$ is recurrent to a compact set
$K \subset \mathcal H^{(1)}_g$ of the moduli space of unit area Abelian differentials, then the horizontal flow $\phi^X_\R$ of the translation surface $(M,h)$  is
uniquely ergodic. 
\end{theorem} 
 \begin{proof}  Let $\mu$ be any probability $\phi^X_\R$-invariant measure on $M$ and let $F(\mu) \in H^1(M,\R)$ denote
 its flux class. As in the proof of Claim \ref{claim:UE}, by the ergodic theorem the flux class is given by the formula
 $$
 F(\mu) = \lim_{T\to +\infty}  \frac{1}{T}  \gamma_T^X (p) \,.
 $$
For all $t>0$, the segment $\gamma_T^X(p)$ has length  $ e^{-t} T$ on the translation surface $g_t (M, h)$, hence for any 
compact set $K$ there exists a constant $C_K>0$ such that, whenever $g_t (M,h) \in K$, the  norm of  the current $\gamma_T^X(p)$ in the dual Sobolev space $W^{-1}(g_t(M,h))$ of the (flat) translation surface  $g_t (M,h)$ satisfies (for 
$T >  e^t$) the bound
$$
\Vert  \gamma_T (p) \Vert_{W^{-1}(g_t(M,h))}  \leq  C_K  e^{-t}  T \,.
$$ 
Since over $K$, all norms on the cohomology bundle are equivalent we derive the following bound for the Hodge norm of the
flux class: there exists a constant $C'_K>0$ such that,  for all $t >0$, we have 
$$
\Vert g^{\text{\rm KZ}}_t (M, h, F(\mu) ) \Vert  \leq  C'_K   e^{-t}  \,.
$$
Under the recurrence hypothesis, the above estimate implies that $F(\mu) \in \mathcal T(M,h)$. In fact, if  $F(\mu)$ has
a component $F_0 (\mu) \not =0$ in $ \mathcal T(M,h)^\perp$, by Lemma~\ref{lemma:spectral_gap}  there exists a constant $C''_K>0$ such that
\begin{equation}
\label{eq:fast_exp_decay}
\begin{aligned}
\Vert  (M, h, F_0(\mu) ) \Vert &\leq   \Vert g^{\text{\rm KZ}}_t (M, h, F_0(\mu) ) \Vert  \exp \big(\int_0^t  \Lambda (g_s(M,h)) ds \big)
\\ &\leq   C''_K   \exp\big ( - t + \int_0^t  \Lambda (g_s(M,h)) ds \big)
\end{aligned}
\end{equation} 
The assumption  that $F_0(\mu) \not =0$ hen implies that
$$
 \limsup_{t\to +\infty}  \left(   t - \int_0^t  \Lambda (g_s(M,h)) ds \right)    <  + \infty\,,
 $$
in contradiction with the assumption that the forward orbit  $g_{\R^+} (M,h)$ is recurrent to the compact set $K$. In fact,  since 
for any compact set $K'$ such that $K \subset K'$  we have  $ \max _{K'}  \Lambda <1$,  and the set $K'$ can be chose so that,
since the trajectory is recurrent to $K$,  the total Lebesgue measure of the time intervals that it spends in $K'$ is infinite.

\noindent We have thus proved that $F(\mu) \in \mathcal T(M,h)$, and again by the estimate in formula~\eqref{eq:fast_exp_decay},
it follows that $F(\mu) \in \R [ \im(h)]$ (since $[ \im(h)]$ has Lyapunov exponent equal to $1$). Thus by the
injectivity Lemma ~\ref{lemma:inj} the cone of invariant measures is one-dimensional and $\phi^X_\R$ is uniquely ergodic.
\end{proof} 

\noindent The above Masur's criterion  implies, by the Poincar\'e recurrence theorem, that all $g_\R$- invariant probability measures are supported on the set of translation surfaces with uniquely ergodic horizontal and vertical flows. 

\noindent It implies in particular 
 Theorem~\ref{thm:Keane_conj}, which states that the Masur--Veech typical translation flow is   uniquely  ergodic (Keane conjecture).

\smallskip 
\noindent The stronger unique ergodicity result given in Theorem~\ref{thm:Keane_conj}, which states  that unique ergodicity is directionally typical for any given translation surface, follows from Masur's criterion and from the statement first proved in \cite{KMS86} that almost all directional flows are recurrent.  

\smallskip 
\noindent An effective (polynomial) unique ergodicity theorem  (Theorem~\ref{thm:eff_erg}) was later derived in \cite{AtF08},
from Lemma \ref{lemma:spectral_gap}  and from the effective  effective version of the recurrence result of \cite{KMS86} proved by J.~Athreya  in his Ph. D. thesis  \cite{At06}. 

\smallskip
In fact, the spectral gap result implies the following conditional effective unique ergodicity result. 

\begin{theorem} 
\label{thm:eff_UE_typ} Assume that the forward Teichm\"uller orbit $g_{\R^+} (M,h)$ visits a compact set $K \subset  \mathcal H^{(1)}_\kappa$ with positive frequency, that is,
$$
f_K := \liminf_{t \to + \infty}  \text{ \rm Leb} ( \{  t \geq 0 \vert   g_t(M,h) \in K \} ) > 0\,.
$$
Then there exist constants $C(M,h) >0$ and $\alpha  >0$ such that, for all functions $f\in H^1(M)$  of zero average
and for all  $(p,T) \in M \times \R^+$, such that $p$ has an infinite forward orbit under $\phi^X_\R$,  we have
$$
\big \vert  \frac{1}{T} \int_0^T  f \circ \phi^X_t (p) dt \big  \vert  \leq  C(M,h) \Vert  f \Vert_{H^1(M)} \,  T^{-\alpha} \,.
$$
\end{theorem} 

The proof of Theorem~\ref{thm:eff_UE_typ} is analogous to that of Theorem~\ref{thm:eff_UE_selfsim} in 
section~\ref{subsec:selfsim}. It is based on the proof of a spectral gap result for a distributional  ``transfer cocycle'', which is 
derived from Lemma~\ref{lemma:spectral_gap} by de Rham theorem. 

\smallskip 
\noindent We give below the precise definition of the Sobolev bundle of currents and of the transfer cocycle.

\smallskip
\noindent  Let $(M,h)$ be a translation surface with horizontal and vertical vector fields $(X,Y)$. Let $W^1(M,h)$ denote the  Sobolev space of $1$-forms on $M$ endowed with the Sobolev norm induced by the flat metric of $(M,h)$. The Sobolev norm of the space 
 $W^1(M,h)$ is defined as follows: for every smooth $1$-form $\alpha$ on $M$
$$
\Vert  \alpha \Vert_{W^1(M,h)} := \left(  \Vert  \imath_ X \alpha \vert  \Vert^2_{L^2(M, \text{\rm area}_h)} + 
\Vert  \imath_ Y \alpha \vert  \Vert^2_{L^2(M, \text{\rm area}_h)} \right)^{1/2}\,.
$$
Let $W^{-1}(M,h)$ denotes the dual Sobolev space of currents of dimension $1$ (and degree $1$) on $M$ endowed with 
the Sobolev norm induced by the flat metric of $(M,h)$. 

\begin{definition}  The {\bf Sobolev bundle of $1$-currents}  $W^{-1}_\kappa$  over a stratum ${\mathcal H}^{(1)}_\kappa$ is the quotient of the Sobolev (Hilbert) bundle 
$$
  \bigcup_{(M,h) \in \hat{\mathcal H}^{(1)}_\kappa}  \{ \{(M,h)\}  \times   W^{-1} (M,h) \} \,.
$$
under the action of the mapping class group $\Gamma_g$ on ${\mathcal H}^{(1)}_\kappa$ and, for every $(M,h)\in {\mathcal H}^{(1)}_\kappa$,  on  the Sobolev space $W^{-1} (M,h)$ of currents by push-foward. The bundle $W^{-1}_\kappa$ is a Hilbert
bundle with norm
$$
\Vert  (M, h, \gamma) \Vert_{-1} :=  \Vert  \gamma \Vert_{W^{-1}(M,h)} \,, \quad \text{ for all } (M,h)\in {\mathcal H}^{(1)}_\kappa \text{ and }   \gamma \in W^{-1}(M,h)\,.
$$
The {\bf transfer cocycle} $g_t^{{(1)}}$ on  $W^{-1}_\kappa$ is defined as the quotient of the trivial cocycle
$$
g_t  \times \text{ \rm Id } :   (M,h) \times  W^{-1}(M,h)  \to  g_t(M,h) \times  W^{-1}(g_t(M,h))\,,
$$
under the action of the mapping class group. 
\end{definition} 

\noindent We now state the spectral gap result for the transfer cocycle (analogous of Claim~\ref{claim:spectrum_closed}) in section
\ref{subsec:selfsim}.  

\begin{lemma}  
\label{lemma:spectral_gap_transfer}
The  subbundle  $\mathcal Z^{-1}_\kappa \subset W^{-1}_\kappa$ of closed $1$-currents is invariant under the 
 transfer cocycle $g^{(1)}_\R$ and the restriction of $g^{(1)}_\R$ to $\mathcal Z^{-1}_\kappa$ has a spectral gap, in the following sense.
If the forward Teichm\"uller  orbit $g_\R (M,h)$ visits a compact set $K \subset  \mathcal H^{(1)}_\kappa$ with positive frequency,  in the sense that 
\begin{equation}
\label{eq:f_K}
f_K := \liminf_{t \to + \infty}  \text{ \rm Leb} ( \{  t \geq 0 \vert   g_t(M,h) \in K \} ) > 0\,,
\end{equation}
then,  there exist constants $C>0$ and $\alpha>0$ such that, for all $\gamma \in \mathcal Z^{-1}_\kappa (M,h)$,  such that 
$\<\gamma, \re (h)\>= \<\gamma, \im(h)\>=0$, and for all $t>0$, 
$$
 C^{-1}  e^{-(1-\alpha)t}  \Vert  (M, h, \gamma) \Vert_{-1}\leq   \Vert  g^{(1)}_t (M, h, \gamma) \Vert_{-1} \leq  C   e^{(1-\alpha)t}  \Vert (M, h, \gamma) \Vert_{-1} \,.
$$
\end{lemma}  
\begin{proof} The argument is analogous to the one given in the proof of Claim~\ref{claim:spectrum_closed} and will generalized
in the proof of the effective Veech criterion (see Theorem \ref{thm:effective_criterion}) in the part on effective weak mixing. We briefly summarize  the main steps here.   

\noindent It is based on the de Rham theorem for currents, on the spectral gap Lemma~\ref{lemma:spectral_gap} for the cohomology bundle and on a direct proof that the (invariant) subbundle of exact currents carries the single Lyapunov exponent $0$ (with infinite multiplicity).

\noindent The latter statement in turn follows from the basic fact that exact currents in $W^{-1}(M,h)$ are exterior derivatives of square-integrable functions, and that the norm on $L^2(M,h)$ is invariant under the Teichm\"uller flow, in the sense that
$$
\Vert  u  \Vert_{L^2( g_t(M,h))} =  \Vert  u  \Vert_{L^2(M,h)}  \,, \text{ for all }  u\in L^2(M, h) \text{ and }  t\in \R\,,
$$
since the area form $\text{\rm area}_h$ is invariant under the Teichm\"uller  flow.  

\end{proof} 

\begin{proof} [Proof of Theorem~\ref{thm:eff_UE_typ}] 
We outline the argument which, as announced, is similar to the proof of Theorem~\ref{thm:eff_UE_selfsim}. However, the point
of view is different and geometric, as we prove below bounds on a given orbit segment under the Teichm\"uller deformation of the flat metric and of the related Sobolev norms, while in the proof of the self-similar case (Theorem~\ref{thm:eff_UE_selfsim}) we proved bounds on the push-forwards of a given orbit segment under the action of the iterates of a pseudo-Anosov map.

\smallskip
\noindent We again write ergodic integrals of the horizontal translation flow in terms of one-dimensional currents $\gamma^X_T(p) \in W^{-1}(M,h)$, as in formula \eqref{eq:erg_int}. It is therefore enough to prove bounds on the currents $\gamma^X_T(p)$ with respect to the dual Sobolev norm on the space $W^{-1} (M,h)$. 

\smallskip
\noindent We consider a sequence $(t_n)$ of return times of the forward Teichm\"uller orbit $g_{\R^+}$ to $K$ with positive frequency, that is, such that 
\begin{equation}
\label{eq:return_f} 
\lim_{n\to +\infty}  \,\frac{t_n}{n}  = f   >0 \,.
\end{equation}
Let $\gamma^X_{T_n}(p)$ be the first foward return orbit of the horizontal flow to a vertical interval centered at $p\in M$
of unit length on $g_{t_n} (M,h)$.  We note that a vertical interval of unit length on $g_{t_n} (M,p)$ has length 
$e^{-t_n}$ on $(M,h)$ and that $\gamma^X_{T_n}(p)$ is the also the first forward return orbit 
$\gamma^{X_{t_n}} _{e^{-t_n}T_n}(p)$ for the horizontal flow $\phi^{X_{t_n}}_\R$ of $g_{t_n} (M,h)$.  Since $g_{t_n} (M,h) \in K$, there exists a constant $C_K>0$ such that
$$
e^{t_n}/ C_K   \leq   T_n(p)   \leq  C_K  e^{t_n} \,.
$$
By a standard decomposition argument (a generalization of the so-called Ostrowski expansion of an irrational number), 
by the condition in formula~\ref{eq:return_f},  it is then possible to reduce estimates on ergodic integrals for arbitrary times to estimates for  the sequence of times $(T_n(p))$  defined above. We then estimate the Sobolev norm of the currents  $\gamma^X_{T_n}(p) \in W^{-1}(M,h)$. 

\smallskip
\noindent  By its definition as a return orbit to a vertical interval, the horizontal orbit segment $\gamma^X_{T_n}(p)$  can be closed by the union $\bar \gamma^X_{T_n}(p) := \gamma^X_{T_n}(p) \cup I_n$  with a vertical segment $I_n$ of uniformly bounded length on $g_{t_n}(M,h)$ connecting its endpoints.  Since $g_{t_n}(M,h)\in K$, by definition the loop $\bar \gamma^X_{T_n}(p)$ has bounded length on $g_{t_n}(M,h)$ and  there exists a constant $C'_K>0$ such that 
$$
\Vert  g^{(1)} _{t_n} ( M, h, \bar \gamma^X_{T_n}(p)) \Vert_{-1}  \leq  C'_K, \quad \text{ for all }n \in \N\,.
$$
Since $\bar \gamma^X_{T_n}(p)) \in \mathcal Z^{-1} ( g_{t_n} (M,h))$ is a closed current, by the spectral gap Lemma~\ref{lemma:spectral_gap_transfer} for the transfer cocycle, there exist $C''_K>0$ and  $\alpha >0$ such that we have 
$$
\Vert (M, h, \bar \gamma^X_{T_n}(p)) \Vert_{-1} \leq  C''_K   e^{ (1-\alpha) t_n}  \,, 
$$ 
Since the {\it vertical }interval $I_n$ has length at most $e^{-t_n}$ on $(M,h)$  (by construction it has at most unit  length on
$g_{t_n} (M,h)$), we conclude that 
$$
\Vert  \gamma^X_{T_n}(p) \Vert_{W^{-1}(M,h)} \leq   C''_K  T_n^{1-\alpha}\,,
$$
which completes the proof of the lemma (up to the Ostrowski-type decomposition result, for which we refer to 
\cite{F02}, Lemma 9.4). 
\end{proof}

\noindent Finally, the polynomial unique ergodicity Theorem \ref{thm:eff_erg} for directionally typical translation flows follows from Lemma~\ref{thm:eff_UE_typ} and from the result of J.~Athreya's \cite{At06}. In fact,  Athreya's results imply  that the 
positive frequency condition in formula~\eqref{eq:f_K} holds for Lebesgue almost all directions on {\it every }translation 
surface (that is, for the horizontal translation flow of $(M, e^{\imath \theta} h)$, for {\it every }translation surface $(M,h)$ 
and for Lebesgue  almost all $\theta \in \T$).

\section{(Effective) weak mixing: a twisted cohomology approach}

\noindent The cohomological approach to weak mixing of translation flows is based on a notion of twisted cohomology, which arises naturally if we want to attach cohomology classes to eigenfunctions, generalizing the flux cohomology 
class of an invariant function (or an invariant measure).  It  emphasizes a close analogy between (effective) unique
ergodicity and (effective) weak mixing, viewed as the (effective) unique ergodicity of twisted flows (the products with linear flows
on the circle). The relevant twisted cohomology has a well-Hodge decomposition and a Hodge norm, and it is possible to
compute once again the first variation of the Hodge norm, and derive a spectral gap result for the corresponding twisted cocycle.
Bounds on twisted integrals of translation flows can then be derived from the spectral gap of the twisted cocycle, in analogy
with the (untwisted) case of ergodic averages.

\subsection{Twisted cohomology} 
\label{subsec:Twisted_Cohom}

\noindent We associate to every eigenfunction of a translation flow a twisted cohomology class. 

\smallskip
\noindent Let $(X,Y)$ denote the horizontal and vertical flow of a translation surface $(M,h)$.  Let $u\in L^2(M, \text{\rm area}_h)$
be an eigenfunction of the horizontal translation flow $\phi^X_\R$ of eigenvalue $-2\pi i \lambda$ (with $\lambda \in \R$):
$$
u \circ \phi^X_t = e^{-2\pi i \lambda t} u\,,   \quad \text{ for all } t \in \R \,, \quad \text{ or } \quad   X u+2\pi i \lambda u =0 \,.
$$
The $1$-form $u \im(h)$ is closed for the twisted differential  
$$ 
d_{h, \lambda} =  d+ 2\pi \imath \lambda \re (h)  \wedge \,, 
$$
in fact, since $h$ is $d$-closed, 
$$
\begin{aligned}
d_{h, \lambda} ( u  \im (h)) & = d u \wedge  \im(h) +   2\pi \imath \lambda u  \re(h) \wedge \im (h) \\ &=  (Xu \cdot  \re (h) + Yu \cdot \im(h)) \wedge  \im(h) + 2\pi \imath \lambda u \cdot \re (h) \wedge \im(h) \\ &=  (Xu + 2\pi i \lambda u ) \text{\rm area}_h =0\,.
\end{aligned} 
$$
In the above calculation, all derivatives are taken in the weak $L^2$ sense. 

\begin{exercise} a) Prove that in the case $\lambda=0$,  square integrable functions with zero cohomology class are 
distributional derivatives of continuous invariant functions, hence they vanish  identically if the horizontal 
translation flow is minimal.  

\noindent b) Prove also that, for $\lambda \not =0$, square integrable functions with zero twisted 
cohomology class are  distributional derivatives of continuous invariant functions. 

\noindent c) Prove that non-trivial continuous 
eigenfunctions for the horizontal flow exist if and only if the translation surface is a (translation) cover of the circle 
$\T = \R/\Z$,   if and only if  it has a completely periodic directional foliation, with cylinders of commensurable heights.
\end{exercise} 

\noindent We generalize the above definition of twisted differential to arbitrary closed $1$-forms and we introduce the
corresponding twisted cohomology.

\begin{definition}
\label{def:twisted_diff_cohom} 
For any closed $1$-form $\eta$, the  {\bf twisted cohomology} $H^1_\eta(M, \mathbb C)$ is defined as the cohomology of the 
{\bf twisted differential } $$d_\eta= d+ 2\pi \imath \eta\wedge \,,$$
that is, the cohomology of the differential complex $(\Omega^\ast(M), d_\eta)$ of differential forms, which is defined as
as the quotient
$$
H^1_\eta (M, \C) :=  \frac{ \text{\rm Ker} (d_\eta: \Omega^1(M,\C) \to \Omega^2(M, \C))}{ \im(d_\eta: \Omega^0(M, \C) \to \Omega^1(M))   } \,.
$$
\end{definition} 
\noindent We note that the above formula gives a well-defined differential operator since for all smooth differential form $\alpha \in \Omega^\ast(M, \C)$ we have, since $\eta$ is a closed $1$ -form
$$
d^2_\eta \alpha=  (d+ 2\pi \imath \eta\wedge) \circ  (d+ 2\pi \imath \eta\wedge) \alpha = d^2 \alpha 
+ 2\pi \imath \eta\wedge d\alpha  - 2\pi  i \eta \wedge d\alpha = 0\,.
$$
It follows that $(\Omega^\ast(M), d_\eta)$ is a  well-defined differential complex and the twisted cohomology $H^1_\eta(M, \C)$ is well-defined.

\smallskip
\noindent The twisted differential $d_{h, \lambda}$ defined above is a special case of  the twisted differential $d_\eta$ in  Definition \ref {def:twisted_diff_cohom} when the closed $1$-form  $\eta =\lambda \re(h)$.

\begin{lemma} (see \cite{F22a}, Lemma 4.2) For every closed $1$-form $\eta$ on $M$, the twisted cohomology $H^1_\eta(M, \mathbb C)$ only depends,
up to a unitary (gauge) transformation, on the cohomology class $[\eta] \in H^1(M, \R)/ H^1(M, \Z)$. 
\end{lemma} 
\begin{proof} If $[\eta_1] = [\eta_2] \in  H^1(M, \mathbb R)/ H^1(M, \mathbb Z)$, then there exists $f :M \to \T$ such that 
$\eta_1-\eta_2 = df$, and, given any point $p \in M$ we can define
$$
f(x) = \int_{p}^x  \eta_1 -\eta_2 \,, \quad \text{ for all }  x\in M\,.
$$
The above definition is well-posed since the value of the integral does not depend modulo $\Z$ on the path of integration
between the points $p$ and $x \in M$.

\noindent We then define the ``gauge transformation'' as
$$
U_f= e^{2\pi \imath f}  : H^1_{\eta_1}(M, \mathbb C) \to H^1_{\eta_2} (M, \mathbb C)\,.
$$
In fact, since $f: M \to \mathbb T$ the function $e^{2\pi \imath f}$ is well-defined on $M$ with modulus equal to $1$
and, for all $1$-form $\alpha \in \Omega^1(M, \C)$ we have
$$
\begin{aligned}
d_{\eta_2} (U_f \alpha) &= d(e^{2\pi \imath f} \alpha) +  e^{2\pi \imath f}  \eta_2 \wedge \alpha  \\ & =
e^{2\pi \imath f}  ( d\alpha +  2\pi i df \wedge \alpha + \eta_2 \wedge \alpha ) \\ &
=e^{2\pi \imath f}  ( d\alpha +  2\pi i (\eta_1-\eta_2) \wedge \alpha + \eta_2 \wedge \alpha )  = U_f  (d_{\eta_1} \alpha) \,. 
\end{aligned} 
$$
Thus the transformation $U_f$ intertwines the twisted differentials $d_{\eta_1}$ and $d_{\eta_2}$ and therefore maps $H^1_{\eta_1}(M, \C)$ onto $H^1_{\eta_2}(M, \C)$.
\end{proof}

\noindent It follows in particular that for $[\eta] \in H^1(M, \Z)$ we have that $H^1_\eta(M, \C)\equiv H^1(M, \C)$, hence it
has dimension $2g$. 

\noindent For $[\eta] \not \in H^1(M, \Z)$, it can be proved that the twisted cohomology $H^1_\eta(M, \C)$ has dimension $2g-2$ (prove as an Exercise or see \cite{F22a}, Lemma 4.3).

\noindent  This dimension loss can be accounted for since $[\eta] \in H^1(M, \Z)$ if and only if $H^0_\eta(M, \C)$ is non-trivial, and
in fact it is one-dimensional (prove as an exercise or see \cite{F22a}, Lemma 4.1). 

\subsection{The twisted cocycle}   In Definition \ref{def:Hodge_bundles} we have introduced the real Hodge bundle $H^1_\kappa(M,\R)$ over a stratum $\mathcal H^{(1)}_\kappa$ of the moduli space of translation surfaces.  We introduce below
the toral Hodge bundle and the twisted cohomlogy bundle, on which  the key renormalization cocycles for weak mixing are defined.

\begin{definition}
The {\bf toral Hodge bundle}  $H^1_\kappa(M,\T)$  over $\mathcal H^{(1)}_\kappa$ is the quotient
$$
H^1_\kappa(M,\T):=  H^1_\kappa(M,\R)/ H^1_\kappa(M,\Z)
$$
of the real Hodge bundle over the sub-bundle with fibers given by the integral cohomology.
\smallskip
\noindent The {\bf twisted cohomology bundle} $\mathcal T^1_{\kappa} (M, \C)$ is the bundle over $H^1_\kappa(M,\T)$ with fibers given by the twisted cohomology spaces, defined as the quotient of the bundle
$$
\hat {\mathcal T} ^1_\kappa (M,\ C)  = \{ (h, \eta, c)   \vert   (h, \eta)  \in \hat {\mathcal H}^{(1)}_\kappa  \times H^1(M, \R)
\quad \text{and} \quad  c \in H^1_\eta(M,\C)  \}/ \Gamma_g $$
with respect to the action of the group $H^1(M, \Z)$ on $H^1(M,\R)$ by translations and on $H^1_\eta(M,\C)$  
by unitary transformations. 
\end{definition} 

\noindent Parallel transport yields cocycles over the Teichm\"uller flow $g_\R$ and, more generally, over the the $SL(2, \mathbb R)$ action on the moduli spaces of translation surfaces

\begin{definition}  The {\bf twisted cocycle}  (denoted as $g^\#_\R$) is the lift to the twisted cohomology bundle 
of the toral Kontsevich--Zorich cocycle (the projection of the Kontsevich--Zorich cocycle onto the toral Hodge bundle 
$H^1_\kappa(M, \T)$) by parallel transport: for all  $[(h, \eta, c)] \in \mathcal T^1_{\kappa} (M, \C)$ and for all $t \in \R$, we define
$$
g^\#_t [(M, h, \eta, c)]  =  [ (g_t(M, h), \eta,c)  ]  \,. 
$$
In the above formulas the symbol $[(h, \eta, c)]$  denotes the equivalence class of the triple $(h, \eta, c)$ such that
$ (h, \eta) \in \in \hat {\mathcal H}^{(1)} \times H^1(M, \R)$ and $c \in H^1_\eta(M, \C)$   with respect to the action of the mapping class group $\Gamma_g$  by pull-back  and to the action of the group $H^1(M,\Z)$ on the space $\{ (\eta, c) \vert  \eta \in H^1(M, \R) \text{ and } c\in H^1_\eta (M, \C)\}.$
\end{definition}

\subsection{The twisted Hodge norm and its first variation} 

The twisted cohomology, hence the twisted cohomology bundle can be endowed with a {\it Hodge norm} and a variational formula can be computed.   Any (closed) real $1$-form  $\eta$ on a translation surface $(M,h)$ (in fact, on any Riemann surface has a Hodge
decomposition into holomorphic and anti-holomorphic part, that is,  $\eta = \eta^{1,0} + \eta^{0,1}$, hence we have the decomposition
$$
d_\eta = d_\eta^{1,0}  + d_\eta^{0,1} =  d^{1,0} +   2\pi \imath \eta^{1,0}  +  d^{0,1} +   2\pi \imath \eta^{0,1}  
$$
Let $\mathcal H ^{1,0}_\eta(M,\C)$  and $  {\mathcal H}^{0,1} _\eta(M,\C)$  denote the space of twisted holomorphic, respectively anti-holomorphic,  differentials:
$$
\begin{aligned}
\mathcal H ^{1,0}_\eta(M,\C)&:= \{\alpha \in \Omega^{1,0} (M, \C) \vert   d_\eta^{1,0} \alpha =0\} \,, \\
\mathcal H ^{0,1}_\eta(M,\C)&:= \{\alpha \in \Omega^{0,1} (M, \C) \vert   d_\eta^{0,1} \alpha =0\} \,.
\end{aligned} 
$$
Since $\eta$ is a real form, and $h$ is of type $(1,0)$ (holomorphic)  there exists a smooth function $f_\eta: M \to \C$ such that 
$$
\eta^{1,0}  =  f_\eta h  \quad \text{ and }  \quad  \eta^{0,1}  =  \overline{f_\eta} \bar h\,,
$$
hence we have twisted Cauchy--Riemann operators
\begin{equation}
\label{eq:CR_op} 
\begin{aligned} 
{\partial}^+_\eta  = \partial^+ + 2 \pi \imath   \overline{ f_\eta} = (X+\imath Y)  + 2 \pi \imath \overline{ f_\eta} \,, \\ 
{\partial}^-_\eta  = \partial^- + 2 \pi \imath  f_\eta = (X-\imath Y)  + 2 \pi \imath  f_\eta\,.
\end{aligned} 
\end{equation}

\begin{exercise} Given a real $1$-form $\eta$,  prove that every  twisted-holomorphic (resp. anti-holomorphic) differential $\alpha$  (for the $1$-form $\eta$) can be written as  $\alpha = m^+  h$ (resp.$\alpha = m^-  \bar h$)  with $m^+$ (resp. $m^-$) a twisted-holomorphic (resp. anti-holomorphic)  function, that is, such that ${\partial}^+_\eta m^+=0$  (resp. ${\partial}^-_\eta m^- =0$).
\end{exercise} 

\noindent We note that, for any real closed $1$-form  $\eta \in \Omega^1(M, \R)$  the twisted cohomology space 
$H^1_\eta (M, \C)$ has no real subspace, hence for real closed $1$-form $\eta$ we defined the real twisted cohomology 
$$
H^1_\eta (M, \R) :=   \re( H^1_\eta (M, \R) \oplus  H^1_{-\eta}  (M, \R) )\,.
$$
It can be proved that any class $c \in H^1_\eta (M, \R) $ has a twisted holomorphic representative, that is,  there exists 
a  $1$-form $\alpha_\eta + \alpha_{-\eta} \in \mathcal H^{1,0}_\eta \oplus \mathcal H^{1,0}_{-\eta}$ such that 
$$
c= \re ( [ \alpha_\eta]  +  [ \alpha_{-\eta}] ) \in  H^1_\eta(M,\C) \oplus H^1_{-\eta}(M,\C)\,.
$$ 
\begin{definition} Let $\eta$ be a real closed $1$-form.  The {\bf twisted Hodge norm} of any  twisted cohomology class 
$c \in H^1_\eta(M, \C)$  is defined in terms of a $d_\eta$-closed representative $\alpha$ as follows:
$$
\Vert  c \Vert_{H^1_\eta (M, \C)} := \left( \frac{\imath}{2} \int_M  \alpha \wedge \bar \alpha \right)^{1/2}\,.
$$
The {\bf twisted Hodge norm} of  a  real twisted cohomology class $c \in H^1_\eta(M,\R)$  is defined in terms of a twisted holomorphic representative $\alpha_\eta + \alpha_{-\eta}$ as
$$
\Vert c  \Vert_{H^1_{\eta}(M,\R)} := \left( \Vert  \alpha_\eta \Vert^2 _{H^1_\eta (M, \C)} + \Vert  \alpha_{-\eta} \Vert^2_{H^1_{-\eta} (M, \C)} \right)^{1/2}\,.
$$
\end{definition} 
\noindent {\bf Notation}: For the Hodge norm on the real twisted cohomology bundle  we adopt the notation
$$
\Vert (M, h,\eta, c) \Vert =  \Vert c  \Vert_{H^1_{\eta}(M,\R)} \,, \quad \text{ for all } (h, \eta,c ) \in \mathcal T^1_\kappa(M, \C) \,.
$$
The following variational formula, analogous to the first variational formula in Lemma~\ref{lemma:first_var},  holds. 
The argument follows the original one in the proof of Lemma~5.2  of \cite{F22a} with additional details.

\begin{lemma} 
\label{lemma:first_var_twist} 
For all $(M, h, \eta) \in H^1_\kappa(M, \R)$ and $c \in H_\eta^1 (M, \R)$,  we have 
$$
\frac{d}{dt} \Vert   g^\#_t (h,\eta, c)  \Vert^2 \big\vert_{t=0}   = 4  \,\re \left(  \int_M  \left(\frac{\alpha_\eta}{h}\right)
 \left( \frac{\alpha_{-\eta}}{h} \right) \text{\rm area}_h   \right) :=  2 \,\re B_{(M,h,\eta)} (c)  \,.
$$
\end{lemma} 
\begin{proof}
For $t\in \R$, let us denote $g^\#_t(M, h) = (M_t, h_t)$ and let 
\begin{equation}
\label{eq:first_var_twist_1}
c = \re [  m_{\eta, t}  h_t  +  m_{-\eta, t} h_t ] \,,  \quad \text{ with }  m_{\pm \eta, t}  
h_t \in \mathcal H^{1,0}_{\pm \eta} (M_t, \C)  \,.
\end{equation} 
By the definition of the twisted cocycle and of the Hodge norm, we have
$$
\Vert g^\#_t(M,h, \eta, c) \Vert^2 = \left( \int_M (\vert m_{\eta, t} \vert^2 +  \vert m_{-\eta, t} \vert^2 ) \text{ \rm area}_h  \right)\,,
$$
hence
\begin{equation}
\label{eq:first_var_twist_2}
\frac{d}{dt} \Vert g^\#_t(M,h, \eta, c) \Vert^2  = 2 \re \left( \int_M  \frac{d m_{\eta, t} }{dt} \cdot \overline{ m_{\eta, t} } 
+  \frac{d m_{-\eta, t} }{dt} \cdot \overline{ m_{-\eta, t} }  ) \text{ \rm area}_h  \right)\,.
\end{equation}
By differentiating in formula \eqref{eq:first_var_twist_1}, since the cohomology  class $c \in H^1_\eta(M,\C) \oplus H^1_{-\eta}(M,\C)$ does not depend on $t\in \R$,  since $dh_t/dt  = - \bar h_t$  by the definition of the Teichm\"uller flow (see formula \eqref{eq:h_t_der}) and by taking into account that $d_{-\eta} = \bar d_\eta$, 
there exists a unique smooth function  $w_{\eta,t}$ of zero average such that 
\begin{equation}
\label{eq:first_var_twist_3}
\re \big(  \frac{dm_{\eta, t}}{dt} h_t -   m_{\eta, t} \bar h_t     +  \frac{dm_{-\eta, t}}{dt} h_t -   m_{-\eta, t} \bar h_t  - d_\eta w_{\eta,t} \big)  \equiv 0\,.
\end{equation}
Now we observe that the $1$-forms $\frac{dm_{\pm \eta, t}}{dt} h_t$  are $d_{\pm \eta,t}^{1,0}$-closed,  and the $1$-form 
$m_{\pm \eta, t} \bar h_t $ are $d_{\pm \eta,t}^{0,1}$-closed, hence there exist $\alpha^{1,0}_{\pm \eta, t} \in \mathcal H^{1,0}_\eta(M_t, \C)$ and $\alpha^{0,1}_{\pm \eta, t} \in \mathcal H^{0,1}_\eta(M_t, \C)$, and  smooth functions $u_{\pm \eta,t}$, $v_{\pm \eta,t}$  such that  
$$
\frac{dm_{\pm \eta, t}}{dt} h_t  =  \alpha^{1,0}_{\pm \eta, t}  + d^{1,0}_{\pm \eta,t} u_{\pm \eta,t}  \quad \text{ and } \quad 
m_{\pm \eta, t} \bar h_t  =  \alpha^{0,1}_{\pm \eta, t}  + d^{0,1}_{\pm \eta,t} v_{\pm \eta,t} 
$$
The functions $u_{\pm \eta,t}$, $v_{\pm \eta,t}$  can be chosen so that 
$$u_{\pm \eta,t} \in  \text{\rm ker} (d^{1,0}_{\pm \eta,t} )^\perp = \text{\rm ker} (\partial^-_{\pm \eta,t} )^\perp  \quad \text{ and } \quad v_{\pm \eta,t}  \in \text{\rm ker} (d^{0,1}_{\pm \eta,t} )^\perp = \text{\rm ker} (\partial^+_{\pm \eta,t} )^\perp   \,.
$$
In addition, since $m_{\pm \eta, t} h_t$  are $d_{\pm \eta}$-closed, and the twisted differential does not depend on $t\in \R$, 
 it follows that their derivatives are $d_\eta$-closed, so that
$$
\begin{aligned} 
d_{\pm \eta} \big( \frac{dm_{\pm \eta, t}}{dt} h_t -   m_{\pm \eta, t} \bar h_t  \big)  &= 
d_{\pm \eta}\big(  \alpha^{1,0}_{\pm \eta, t}   -   \alpha^{0,1}_{\pm \eta, t}  +   d^{1,0}_{\pm \eta,t} u_{\pm \eta,t}  - d^{0,1}_{\pm \eta,t} v_{\pm \eta,t}   \big)  \\ &= d_{\pm \eta} ( d^{1,0}_{\pm \eta,t} u_{\pm \eta,t}  - d^{0,1}_{\pm \eta,t} v_{\pm \eta,t} ) =0 \,, 
\end{aligned} 
$$
from which we claim that we can derive that $u_{\pm \eta,t} + v_{\pm \eta,t}  \equiv 0$, hence 
$$
d^{1,0}_{\pm \eta,t} u_{\pm \eta,t}  - d^{0,1}_{\pm \eta,t} v_{\pm \eta,t}  = d_{\pm \eta} u_{\pm \eta,t} \,.
$$
In fact,  we prove below our claim that for all $(M,h)$, for all real harmonic form $\eta$ on $M$ and for  functions $w \in H^1(M)$ 
such that $w \in \bigl( (\text{\rm ker} (\partial^+_{\pm \eta} ) \cap \text{\rm ker} (\partial^-_{\pm \eta} )\bigr )^\perp$ , we have
$$
 d^{1,0}_{\pm \eta}  d^{0,1}_{\pm \eta} w = 0   \Longrightarrow    w=0 \,,
$$
The identity $ d^{1,0}_{\pm \eta}  d^{0,1}_{\pm \eta} w = 0 $ can be written  in terms of the twisted Cauchy-Riemann operators
(since the function $f_\eta$ in formula \eqref{eq:CR_op} is meromorphic) as 
$$
\partial^+_{\pm \eta} \partial^-_{\pm \eta}  w  = \partial^+_{\pm \eta} \partial^-_{\eta}  w=\partial^+ \partial^- w    +  2\pi \imath ( f_{\pm \eta}  \partial^+  + \overline{f_{\pm \eta}}  \partial^-)w - 4 \pi \vert f_{\pm \eta} \vert^2 w   =0\,.
$$
Since the adjoint $(\partial^\pm_{\pm \eta} )^\ast =- \partial^\mp_{\pm \eta}$  as densely defined closed operators on $L^2(M, \text{\rm area}_h)$ with domain the Sobolev space $H^1(M)$, it follows that
$$
\Vert \partial^\pm_{\pm \eta} w \Vert^2 = -  \<  \partial^\mp_{\pm \eta}  \partial^\pm_{\pm \eta} w, w\> =0\,,
$$
hence $w \in \text{\rm ker} (\partial^+_{\pm \eta} ) \cap \text{\rm ker} (\partial^-_{\pm \eta} )$ and, by the assumption that
it also belongs to its orthogonal, we conclude that $w=0$, as claimed.

\smallskip
\noindent  From formula~\eqref{eq:first_var_twist_3} we then conclude that $ \re ( d_{\pm \eta} (u_{\eta,t} + u_{-\eta,t}) - d_{\pm \eta} w_{\pm \eta,t}\equiv  0$ and 
$$
 \re \big(  \alpha^{1,0}_{ \eta, t}  + \alpha^{1,0}_{- \eta, t}    -   \alpha^{0,1}_{ \eta, t}  -   \alpha^{0,1}_{ -\eta, t}   \big) \equiv 0
$$
and from the above equation we derive
$$
\alpha^{1,0}_{ \pm \eta, t}  = \overline{ \alpha^{0,1}_{ \mp \eta, t}     }   \,.
$$
By taking into account that the cup $d^{1,0}_{\pm \eta, t} C^\infty(M)$ is Hodge orthogonal to 
$\mathcal H^{1,0}_{\pm \eta} (M_t, \C)$ and $d^{0,1}_{\pm \eta, t} C^\infty(M)$ is Hodge orthogonal to 
$\mathcal H^{0,1}_{\pm \eta}(M_t, \C)$  follows that
$$
\begin{aligned}
&\re \int_M \frac{dm_{\pm \eta,t }}{dt} \overline{ m_{\pm \eta, t} }  \cdot \text{ \rm area}_h =  
-\frac{1}{2} \im\int_M (\alpha^{1,0}_{\pm \eta, t}  + d^{1,0}_{\eta,t} u_{\pm \eta}) \wedge  \overline{ m_{\pm \eta, t} h_t}  \\ 
&=  -\frac{1}{2} \im \int_M (\alpha^{1,0}_{\pm \eta, t}  ) \wedge  \overline{ m_{\pm \eta, t} h_t } =  
\frac{1}{2} \im \int_M (\alpha^{0,1}_{\mp \eta, t}  ) \wedge  m_{\pm \eta, t} h_t  \\
&=\frac{1}{2} \im \int_M  (m_{\mp \eta, t} \bar h_t  -  d^{0,1}_{\eta,t} v_{\pm \eta} ) \wedge  m_{\pm \eta, t} h_t = 
\re \int_M  (m_{\mp \eta, t}  m_{\pm \eta, t})  \text{ \rm area}_h\,.
\end{aligned} 
$$
By formula \eqref{eq:first_var_twist_2} we conclude that
$$
\frac{d}{dt} \Vert g^\#_t(M,h, \eta, c) \Vert^2  = 4 \re  \int_M  (m_{\eta, t}  m_{- \eta, t})   \text{ \rm area}_h
$$
as stated, hence the argument is completed.
\end{proof} 

\noindent We now derive from the above variation formula a ``spectral gap'' result for the twisted cocycle. Strictly speaking, since
the twisted cocycle does not have ``tautological'' exponents ( the $\pm 1$ exponent in the untwisted case) we will prove
below an upper bound on the top exponent of the twisted cocycle. 

\smallskip
\noindent For all $(M,h) \in \mathcal H^{(1)}_\kappa$, all $\eta \in H^1(M,\R)$ , we let
\begin{equation}
\label{eq:Lambda_twist}
\Lambda^\# (M,h,\eta) :=  \max_{c\in H^1_\eta(M, \R)\setminus\{0\}} \frac{ \vert  B_{(M,h, \eta)}(c) \vert} { \Vert (M,h, \eta, c) \Vert^2}  \,.
\end{equation}

\noindent From Lemma~\ref{lemma:first_var_twist}  we immediately derive

\begin{lemma}
\label{lemma:spectral_gap_twist}
For all $(M,h, \eta) \in H^1_\kappa(M, \T)$, for all $c \in H^1_\eta(M,\R)$ and for all $t\in \R$, we have
$$
\Vert  g^\#_t (M, h, \eta, c) \Vert  \leq   \Vert (M, h, \eta, c) \Vert    \exp \left( \int_0^t  \Lambda^\# \big( g^{KZ}_s (M, h, \eta) \big)  ds \right)\,.
$$
\end{lemma} 

\noindent We prove below a key statement about the function $\Lambda^\#$ analogous to Lemma~\ref{lemma:fund_form} in section~\ref{subsec:Hodge_norm}. The main difference is that while for the function of Lemma~\ref{lemma:fund_form} we have 
$\Lambda<1$ on compact sets,  the identity $\Lambda^\#=1$ is verified on compact set on integral cohomology classes. 

\begin{lemma} (\cite{F22a}, Lemma 5.4) 
\label{lemma:Lambda_twist}
The function $\Lambda^\# $ is well-defined on the toral Hodge bundle $H^1_\kappa(M, \T)$, it satisfies 
the inequality  $0\leq \Lambda^\# \leq 1$ everywhere and 
$$
\Lambda^\# (M,h,\eta) =1  \Longrightarrow  [\eta]  \in H^1(M, \Z) \,.
$$
\end{lemma} 
\begin{proof} The bound follows from the definition and from the Cauchy-Schwarz inequality.  The condition $\Lambda^\# (M,h,\eta) =1$ is equivalent to equality in the  Cauchy-Schwarz inequality.  It implies that there exist  functions $m_{\pm \eta} \in  \text{\rm ker} (\partial^+_{\pm \eta}) $ such that $m_\eta \in \C \overline{m_{-\eta}}$, hence 
$$m_\eta \in \text{\rm ker} (\partial^+_\eta) \cap  \text{\rm ker} (\partial^-_{\eta})\,.$$
It follows that 
$$
d m_\eta + 2\pi \imath m_\eta \eta = d_\eta ( m_\eta) =  (\partial^-_\eta m_\eta) h  + (\partial^+_\eta m_\eta) \bar h  =0\,,
$$
hence (see Exercise \ref{ex:periodic})  we have $\vert m_\eta \vert \in \C$ and, since $m_\eta \not =0$ (as $c \not =0$),
\begin{equation}
\eta =   \frac{1}{2\pi \imath}  \frac{dm_\eta}{m_\eta}  \in H^1(M, \Z)\,.
\end{equation} 
We have thus proved that $\Lambda^\# (M,h,\eta)=1$ implies $[\eta] \in H^1(M, \Z)$ as stated.
\end{proof} 

\begin{exercise} 
\label{ex:periodic}
Prove that the identity $d u + 2\pi \imath u \cdot \eta=0$  for a function $u\in L^2(M, \text{\rm area}_h)$ implies 
that $u \in C^\infty (M)$, as well as $\vert u \vert^2 \in \C$,  and then that,  if $u \not =0$, 
$$
\eta =   \frac{1}{2\pi \imath}  \frac{du}{u}  \Longrightarrow   \eta \in H^1(M, \Z)\,.
$$ 
\end{exercise} 

\subsection{ The Veech criterion revisited}
\label{ssec:Veech_crit}

\noindent We prove below a version of Veech's criterion~\cite{Ve84} for weak mixing of translation flows and
  state a corresponding effective version.  The argument is based on the first variation Lemma~\ref{lemma:first_var_twist} for 
  the   Hodge norm on the twisted cohomology bundle and has not appeared before. It is analogous to the cohomological 
 Hodge theory proof  of Masur's criterion (see Theorem~\ref{thm:Masur_criterion}).
  
  \smallskip
  \noindent Let $g^{KZ}_\R: H^1_\kappa(M, \T) \to H^1_\kappa(M, \T)$ denote the toral Kontsevich--Zorich cocycle
  over the stratum $\mathcal H^{(1)}_\kappa$ of the moduli space of translation surfaces.
  \begin{theorem}  \label{thm:criterion} Let $\lambda \in \R\setminus \{0\}$ be such that $-2\pi \imath \lambda$ is an eigenvalue for the horizontal translation flow $\phi^X_\R$ of a translation surface $(M,h) \in \mathcal H^{(1)}_\kappa$, with no continuous eigenfunction. Then we have the inclusion 
  $$
  \lim_{t\to +\infty} g^{KZ}_t (M, h, [\lambda \re(h)] )  \subset   [ H^1_\kappa(M, \Z) ] \subset H^1_\kappa (M, \T)\,,
  $$
 that is, every limit of the orbit $g^{KZ}_\R (M,h, [\lambda \re(h)] ) $ belongs to the set of integer 
points of the real Hodge bundle over the moduli space of translation surfaces.
  
  \end{theorem}

\smallskip
\noindent The first step in the argument is given by the following lemma. For any $\lambda \in \C$, let $d_{h, \lambda}$ denote
the twisted differential for the closed $1$-form $\eta:=\lambda \re(h)$, that is,
$$
d_{h, \lambda} :=  d + 2\pi \imath \lambda \re(h) \wedge \,.
$$
and let $H^1_{h, \lambda}(M, \C)$ denote the corresponding twisted cohomology.

\begin{lemma} Let $u\in L^2(M, \text{\rm area}_h)$ be an eigenfunction for the horizontal flow with eigenvalue 
$-2 \pi \imath \lambda$.  Then $u\, \im (h)$ is  a $d_{h, \lambda}$-twisted $1$-form. In addition,
if the twisted class $[u \im(h)]_{h, \lambda} =0$, then $u= YU$ is the vertical derivative of a continuous 
eigenfunction $U\in H^1(M)$ of eigenvalue $-2 \pi \imath \lambda$.
\end{lemma} 

\begin{proof}  Let $\{X,Y\}$ denote the generator of the horizontal  and vertical flows, so that, for all $f\in C^\infty(M)$, 
$$
d f = Xf \,\re(h) + Yf\, \im (h) \,.
$$
By assumption we have (in the weak sense)
$$
(X + 2\pi \imath  \lambda) u=0 
$$
As we have mentioned at the beginning of section~\ref{subsec:Twisted_Cohom}, by a calculation (in the weak sense) we derive
$$
d_{h, \lambda} (u \im(h)) = d (u \im(h)) + 2\pi \imath  \lambda u\, \re(h) \wedge \im(h) = (Xu + 2\pi \imath \lambda u)   
 \text{\rm area}_h = 0\,.
$$
Let us now assume that the twisted cohomology class $[u \im(h)]_{h, \lambda} =0 \in H^1_{h,\lambda}(M,\C)$. There exists a function $U\in L^2(M, \text{\rm area}_h)$ such that
$$
d U+ 2\pi \imath \lambda  U \re(h) = u \im(h)\,.
$$
Since the exterior derivative is an elliptic operator, it follows that $U \in H^1(M)$ and by contraction we have
$$
XU + 2\pi \imath \lambda  U = 0   \quad \text{ and }  \quad YU = u\,.
$$
It follows that $U$ is also smooth in the horizontal direction, hence it is continuous (in fact, it is H\"older) by the
Sobolev trace theorem along vertical trajectories. 
\end{proof} 

\begin{exercise}  Prove that every  function $U\in H^1(M)$ such that $(X + 2\pi \imath \lambda)  U = 0$  is of  H\"older class $1/2$ on the complement of the set $\Sigma_h:= \{h=0\}$.
\end{exercise}

\noindent Let $\partial^\pm_{h, \lambda}$ denote the twisted Cauchy-Riemann operators:
$$
\partial^\pm_{h, \lambda} =   \partial^\pm_h  +2\pi \imath \lambda = (X\pm \imath Y) + 2\pi \imath \lambda\,.
$$
By \cite{F22a}, Prop. 3.2,  the maximal closed extension of the operator $\partial^{\pm}_{h, \lambda}$ is the operator 
$- (\partial^{\mp}_{h, \lambda})^\ast$ and there exist orthogonal  decompositions
$$
L^2(M, \text{\rm area}_h) =   \text{Ran} (\partial ^\pm_{h, \lambda}) \oplus    \text{ Ker } ( (\partial ^\pm_{h, \lambda})^\ast    )
 =   \text{Ran} (\partial ^\pm_{h, \lambda}) \oplus    \text{ Ker } (\partial ^\mp_{h, \lambda}  )\,.
$$
For every $u \in L^2(M, \text{\rm area}_h)$  there exist functions $v^\pm \in \text{ Ker } (\partial ^\mp_{h, \lambda})^\perp \cap H^{1}(M)$ and functions $m^\pm \in  \text{ Ker } ( (\partial ^\pm_{h, \lambda})^\ast) = \text{ Ker } (\partial ^\mp_{h, \lambda})\subset L^2(M, \text{\rm area}_h)$ such that 
we have a decomposition
\begin{equation}
\label{eq:u_dec}
u = \partial^+_{h, \lambda} v^+ + m^- =  \partial^-_{h, \lambda} v^- + m^+\,.
\end{equation}
\begin{lemma}  
\label{lemma:Hodge_rep}
Let $u\in L^2(M, \text{\rm area}_h)$ be an eigenfunction for the horizontal flow with eigenvalue 
$-2 \pi \imath \lambda$. Let us assume that $\lambda \re(h) \not\in H^1(M,\Z)$. 
Then we have the following identity between twisted cohomology classes:
$$
[u \im(h) ]_{h, \lambda} = \frac{1}{2\imath} [ m^+  h - m^- \bar h  ]_{h, \lambda}  \in H^1_{h, \lambda} (M, \C)\,.
$$
\end{lemma} 
\begin{proof} Since $m^\pm \in   \text{ Ker } (\partial ^\mp_{h, \lambda})$, it follows that the $1$-forms 
  $m^+  h$ and $m^- \bar h$  are twisted closed. In fact,
$$
\begin{aligned}
d (m^+ h) + (2\pi \imath  \lambda m^+) \re(h)  \wedge h  = \frac{1}{2} (\partial_h^+ m^+ +2\pi \imath \lambda m^+) 
\bar h \wedge h =0 \,, \\
d (m^- \bar h) + (2\pi \imath \lambda m^-)  \re(h)  \wedge \bar h  = \frac{1}{2} (\partial_h^- m^- 
+2\pi \imath \lambda)   h \wedge \bar h =0 
\end{aligned}
$$
We have the identity
$$
u \im(h) = u \frac{ h - \bar h}{2\imath} = \frac{ (\partial^-_{h, \lambda} v^-) h  -  (\partial^+_{h, \lambda} v^+)  \bar h}{2\imath}     + \frac{m^+ h  -  m^- \bar h}{2\imath} \,.
$$
It follows that the $1$-form $(\partial^+_{h, \lambda} v^+) h  -  (\partial^-_{h, \lambda} v^-)  \bar h$ is $d_{h, \lambda}$-closed, which implies that the $1$-forms  $\partial^+_{h, \lambda} (v^+ +v^- ) h$ is $d_{h, \lambda}$-closed, hence $\partial^-_{h, \lambda} \partial^+_{h, \lambda} (v^+ +v^- ) =0$, or 
$$
\partial^-_{h, \lambda} (v^+ +v^- ) = \partial^+_{h, \lambda}  (v^+ +v^- ) =0\,.
$$
Thus the function $v^+ +v^- \in  \text{ Ker } (\partial ^+_{h, \lambda}) \cap  \text{ Ker } (\partial ^+_{h, \lambda})$ and by assumption
 $$
 v^+ +v^- \in  \left( \text{ Ker } (\partial ^+_{h, \lambda}) \cap  \text{ Ker } (\partial ^+_{h, \lambda}) \right)^\perp
 $$
hence $v^+ +v^-  =0$, and we have
$$
u \im(h) = \frac{1 }{2\imath}  d_{h, \lambda} v^+     + \frac{m^+ h  -  m^- \bar h}{2\imath}  \,.
$$
as stated. The argument is therefore complete.
\end{proof} 

\noindent We then consider the orbit of the twisted cohomology class $[u \im(h)]_{h, \lambda}$ under the  the twisted
cocycle and prove that it decays exponentially (with respect to the Hodge norm) in forward time. 

\begin{lemma} \label{lemma:exp_decay}  Let $u\in L^2(M, \text{\rm area}_h)$ be an eigenfunction for the horizontal flow with eigenvalue $-2 \pi \imath \lambda$.  For all $t\in \R$ we have
$$
\Vert g^\#_t (M, h, [\lambda \re(h)],  [u \im(h)] ) \Vert   \leq   \Vert  u  \Vert_{L^2(M, \text{\rm area}_h)}    e^{-t}\,.
$$
\end{lemma}
\begin{proof}
Let $(M_t,h_t) = g_t (M,h)$. We recall that, by definition of the Teichm\"uller flow, we have, for all $t\in \R$, 
$$
h_t = e^{-t} \re (h) + \imath  e^t \im (h)\,,
$$
and by the definition of the cocycle we also have
$$
\begin{aligned}
g^\#_t (M, h, [\lambda \re(h)],  [u \im(h)] )  &=  (M_t, h_t, [e^t  \lambda \re(h_t)],  [e^{-t} u \im(h_t)] ) \\ &= 
(M_t,h_t,  [\lambda_t \re(h_t)],   [u_t  \im(h_t)] )   \,.
\end{aligned}
$$
Let us now write the decomposition of the function $u$ in formula~\eqref{eq:u_dec}  for $(M_t,h_t, \lambda_t \re(h_t))$: there exist functions $m^\pm_t \in \text{\rm ker} (\partial^\pm_{h_t, \lambda_t})$ and $v^\pm_t  \in H^1(M)$ such that 
$$
u_t = \partial^+_{h_t, \lambda_t} v^+_t + m^-_t= \partial^-_{h_t, \lambda_t} v^-_t  + m^+_t \,.
$$
By Lemma~\ref{lemma:Hodge_rep} we have that, for all $t\in \R$, 
$$
[u_t \im(h_t)]_{h_t, \lambda_t } =  \frac{1}{2\imath} [ m^+_t h_t - m^-_t \bar h_t ]_{h_t, \lambda_t} \,,
$$
hence by the definition of the Hodge norm and by orthogonality
$$
\begin{aligned}
\Vert g^\#_t (M, h, [\lambda \re(h)],  [u \im(h)] ) \Vert  &=  \frac{1}{2}( \Vert m^+_t \Vert_{L^2(M, \text{\rm area}_h)} ^2 
+  \Vert m^-_t \Vert_{L^2(M, \text{\rm area}_h)}^2)^{1/2} \\  &\leq  \Vert  u_t \Vert_{L^2(M, \text{\rm area}_h)} =
e^{-t}   \Vert  u \Vert_{L^2(M, \text{\rm area}_h)} \,.
\end{aligned}
$$
\end{proof} 

\noindent Let $\Lambda^\#: H^1_\kappa (M, \T) \to [0, 1]$ denote the function defined in  formula~\eqref{eq:Lambda_twist}.

\begin{lemma}  \label{lemma:bounded_lag}
Let    $\lambda \in \R\setminus \{0\}$ be such that $-2\pi \imath \lambda$  is an eigenvalue for the
horizontal flow of the translation surface  $(M,h)\in \mathcal H(\kappa)$ which has no continuous eigenfunction.  There exists a constant $C_{h, \lambda}>0$  such that, for all $t>0$, we have
$$
\int_0^t \Lambda^\# (g^{KZ}_s (M, h, [\lambda \re (h)]) ds \,\geq \,   t - C_{h, \lambda}\,, \quad \text{ for all } t\in \R^+.  
$$
\end{lemma} 
\begin{proof}  Let $u\in L^2(M, \text{\rm area}_h)$ denote an eigenfunction of the horizontal flow with eigenvalue
$-2\pi \imath \lambda$.  By Lemma~\ref{lemma:spectral_gap_twist}, we have
$$
\begin{aligned}
\Vert (M, h, [\lambda \re(h)],  [u \im(h)] ) \Vert & \leq  \Vert  g^{KZ}_t(h, [\lambda \re(h)],  [u \im(h)] ) \Vert 
\\&\times \exp \left (\int_0^t   \Lambda^\# ( g^{KZ}_{-s}  g^{KZ}_t (h, [\lambda \re(h)],  [u \im(h)] ) ) ds                   \right) \,,
\end{aligned}
$$
hence by Lemma~\ref{lemma:exp_decay} we have
$$
\begin{aligned}
\Vert (M, h, [\lambda \re(h)],  [u \im(h)] ) \Vert &\leq \Vert u\Vert_{L^2(M, \text{\rm area}_h)}   \\ 
&\times \exp \left ( -t + \int_0^t   \Lambda^\# (g_t (h, [\lambda \re(h)],  [u \im(h)] ) ) ds  \right) 
\end{aligned}
$$
Since by assumption $-2\pi\imath \lambda$ has no continuous eigenfunction, the twisted cohomology 
class $[u \im(h)]_{h, \lambda} \not =0$. The statement follows after taking logarithms in the above 
inequality.
\end{proof}

\begin{proof} [Proof of Theorem \ref{thm:criterion}] 
We argue by contraposition. Suppose there exists a sequence  $(t_n) \subset \R^+$ of positive times  such that 
$$
g^{KZ}_{t_n} (M, h, [\lambda \re(h)] ) \to   (M_0, h_0, [\eta_0]) \not\in H^1(M, \Z)\,. 
$$ 
Since by Lemma~\ref{lemma:Lambda_twist} the strict inequality $\Lambda^\# < 1$ holds on $H_\kappa(M, \R) \setminus H_\kappa(M, \Z)$ ,  and since by \cite{F22a}, Lemma~5.3, the function  $\Lambda^\#$ is continuous, there exists $\Lambda^\#_0<1$ and a  neighborhood $\mathcal U_0$ 
of $(M_0, h_0, [\eta_0])$ such that  $\Lambda^\#\vert \mathcal U_0 \leq \Lambda^\#_0$.   

\noindent Since the forward orbit $g_{\R^+} (M, h, [\lambda \re(h)] )$
spends an infinite time in $\mathcal U_0$, it follows that for any constant $C>0$ there exists $t_C>0$ such 
that for all $t\geq t_C$ we have 
$$
 \int_0^t   \Lambda^\# (g^{KZ}_t (h, [\lambda \re(h)],  [u \im(h)] ) ) ds   \leq t - C\,,
$$
thereby contradicting Lemma~\ref{lemma:bounded_lag}.  The criterion is therefore proved.
\end{proof}

\subsection{Weak mixing of  typical translation flows} 
\label{ssec:typical_wm} 

\noindent The proof of weak mixing for typical translation flows (and interval exchange transformations) is based on the Veech
criterion for weak mixing established in section \ref{ssec:Veech_crit}.   In this section we outline the main steps  of
the proof for translation flows following \cite{AvF07}.   The case of interval exchange transformations is harder and the
proof required a different strategy. 

\smallskip
\noindent The Veech criterion suggests the introduction of the notion of a subspace of  the real Hodge bundle which 
contains all translation flows with non-weakly mixing horizontal flow:

\begin{definition}  
\label{def:weak_stable}
The {\bf weak stable space}  $\mathcal W^s_{(M,h)} \subset H^1(M, \R)$ at a translation surface 
$(M,h) \subset \mathcal H^{(1)}_\kappa$ is the set  defined as follows. Let $\mathcal K$ be a (countable) exhaustion
of $\mathcal H^{(1)}_\kappa$ by compact subset. We then define
$$
\mathcal W^s_{(M,h)}  :=  \cap_{K \in \mathcal K }   \{   c \in H^1(M,\R) \vert   g^{KZ}_t (M, h, c) \to H^1_\kappa (M, \Z) \text{ as }  g_t(M, h) \in K\}\,.
$$
\end{definition} 

\noindent In other terms, after projection onto the toral Hodge bundle $H^1_\kappa (M, \T)$, the weak stable space consists of all 
cohomology classes which converge to the zero section of $H^1_\kappa (M, \T)$ under the Kontsevich--Zorich cocycle 
$g^{KZ}_\R$ as the orbit f the Teichm\"uller flow returns to a compact set. 

\smallskip
\noindent We list below the essential properties of the weak stable space. 

\noindent For almost all $(M,h)$ with respect to any $g_\R$-invariant
probability ergodic measure on $\mathcal H^{(1)}_\kappa$ we have:

\begin{itemize}
\item  $\mathcal W^s_{(M,h)}$ depends only on the forward $g_\R$ trajectory of $(M,h)$, hence it only depends on 
$[\im (h)] \in H^1(M,\Sigma_h, \R)$ which determines a stable leaf of the Teichm\"uller flow $g_\R$ 
\item $\mathcal W^s_{(M,h)}$ is saturated by translates of the stable Oseledets subspaces of  the Kontsevich--Zorich cocycle
and equals the union of all integer translates in case the forward $g_\R$-orbit of $(M,h)$ is relatively compact in $\mathcal H^{(1)}_\kappa$.
\end{itemize} 

\noindent We note that in general we do not expect the weak stable space to be contained in the union of integer translates of the Oseledets
central stable space since in the condition which defines the weak stable space we make no assumption on the behavior of the trajectory of the cohomology class during excursions of the forward Teichm\"uller trajectory $g_{\R^+} (M,h)$ outside a compact
subset of the moduli space.  It is therefore possible that the cohomology class converges to the set $H^1_\kappa (M, \Z)$ without
converging to any given element of it (that is, by jumping during excursion from one integer point to another). 

\smallskip

\noindent The key lemma on the weak stable space (which we state below without proof)  gives a bound on its dimension.

\begin{lemma}  (see \cite{AvF07} Th. A.1) For $\mu$-almost all (Oseledets regular) translation surface $(M, h) \in \mathcal H^{(1)}_\kappa$, with respect to any $g_\R$ ergodic probability measure,  the Hausdorff codimension $\text{\rm H-codim } (\mathcal W_{(M,h)})$ of the weak stable space is greater or equal than the dimension of the unstable Oseledets space $E^+_{(M,h)} \subset H^1(M, \R)$, that is, it satisfies the bound
$$
\text{\rm H-codim } (\mathcal W^s_{(M,h)})    \geq   \text{ \rm dim}  (E^+_{(M,h)} ) =  \sum_{\lambda_i^\mu >0} \text{ \rm dim}  (E_{(M,h)}(\lambda^\mu_i) )\,.
$$
\end{lemma} 

\noindent From the Veech criterion and from above dimension bound, we derive the following

\begin{theorem} \label{thm:lin_elimin}  (Linear elimination)  Let $\mu$ be any $g_\R$-invariant ergodic measure with a product structure
with respect to the invariant foliations of the Teichm\"uller flow and  with unstable dimension $d^u_\mu$ (in the absolute cohomology, that is, after the projection $H^1(M, \Sigma, \C) \to  H^1(M, \C)$ in period coordinates) satisfying the bound
$$
d^u_\mu  >    1 +  \text{ \rm codim}  (E^+_\mu) =  1+ 2g -  \sum_{\lambda_i^\mu >0} \text{ \rm dim}  (E_{(M,h)}(\lambda^\mu_i) )\,,
$$ 
then  for $\mu$-almost all $(M,h) \in \mathcal H^{(1)}_\kappa$ the horizontal translation flow $\phi^X_\R$ is weakly mixing.
In particular,  horizontal translation flow $\phi^X_\R$ is weakly mixing for the Masur--Veech typical translation surface.
\end{theorem} 

\begin{proof}  By the Veech criterion \cite{Ve84} (see also Theorem~\ref{thm:criterion})  if the horizontal flow is weakly mixing 
under the condition that 
$$
\R [ \re(h) ]  \not \in  \mathcal W^s_{(M,h)}   \Longleftrightarrow    [ \re(h) ]  \not \in \R \cdot   \mathcal W^s_{(M,h)}  \,.
$$
For all $(M,h)$ which is an Oseledets regular point for the Kontsevich--Zorich cocycle $g^{KZ}_\R$, the set $\R \cdot   
\mathcal W^s_{(M,h)}$ depends only on $[\im (h)]$ (which in turn determines the stable leaf of the Teichm\"uller flow).
In addition we have the dimension bound
$$
\text{\rm H-dim } (\R \cot \mathcal W^s_{(M,h)}) \leq 1+ \text{\rm H-dim } (\mathcal W^s_{(M,h)}) \leq  1+  \text{ \rm codim}  (E^+_\mu) < d^u_\mu\,.
$$
It follows that, since $\mu$ has a product structure $\mu = \mu_s \otimes \mu_u$, for  $\mu_s$-almost all $[\im(h)] 
\in H^1(M, \Sigma, \R)$ the translation surface $(M,h)$ is Oseledets regular for $g^{KZ}_\R$, and by the above
inequality it follows that, for $\mu_u$-almost all $[\re (h)] \in H^1(M, \Sigma, \R)$,  we have that $\R [\re(h)] \cap \mathcal W^s_{(M,h)} =\{0\}$, hence  the horizontal flow is weakly mixing. 

\smallskip
\noindent In the particular case of the Masur--Veech measure $\mu= \mu^{(1)}_\kappa$, it was proved in \cite{F02} (and later in \cite{AV07}) that $\text{ \rm dim}  (E^+_\mu) =  g$ (the genus). In addition, the Masur--Veech measure has a product structure
with dimension $d^u_\mu = 2g$. Since for $g \geq 2$ we have
$$
2g > 1+ 2g  - g = g+1\,,
$$
it follows that the Masur--Veech typical translation flow is weakly mixing. The argument is complete.
\end{proof} 

\begin{remark}  An analogous argument  applies to canonical invariant measures supported on $\SL(2, \R)$-invariant
orbifolds $\mathcal M$ of rank $r$ at least $2$. In fact, every  canonical invariant measure $\mu$ has a product structure and  the stable and unstable dimensions  $d^u_\mu= d^s_\mu =2r$.  Since $\mathcal M$ is $g_\R$-invariant the analysis can be restricted to the invariant subbundle $H^1_{\mathcal M}(M,\R) \subset H^1_{\kappa}(M,\R)$ given by the real part of the projection of the tangent bundle $T\mathcal M \subset H^1_\kappa(M, \Sigma, \C)$ (the complex relative cohomology bundle) onto $H^1_\kappa(M,  \C)$ (the  complex absolute cohomology bundle).  It is known that the Kontsevich--Zorich exponents are all non-zero on $H^1_{\mathcal M}(M,\R)$ (see \cite{F11} and \cite{Fi17}, Cor. 1.3),  which implies that $\text{ \rm dim}  (E^+_\mu) = r$.  A condition analogous to that of Theorem  \ref{thm:lin_elimin} is therefore verified if and only if 
$$
2r =d^u_\mu  > 1  + 2r -r = r+1 \Longleftrightarrow   r \geq 2\,.
$$
\end{remark}

\begin{remark} The case of $\SL(2,\R)$-invariant orbifolds of rank one is in general open with the important exception of 
closed $\SL(2,\R)$-orbits (Veech translation surfaces). Avila and Delecroix  \cite{AD16} have investigated in depths this case and proved in particular directionally typical weak mixing on all non-arithmetic Veech surfaces. These include surfaces coming by the
unfolding of billiards in regular polygons with at least 5 edges. 
\end{remark} 

\noindent We conclude this section with a few words on the proof of weak mixing for Interval Exchange Transformations (IET's). 
The problem is essentially equivalent to establishing weak mixing for translation flows $(M, h)$ with a {\it fixed}  class $\re(h) \in H^1(M, \R)$ (for almost all $\im(h) \in H^1(M, \R)$).  Indeed, an interval exchange transformation is equivalent to the (horizontal) translation flow obtained by suspension under  a constant roof function. The roof function determines the class $\re(h)$. The linear elimination procedure outlined above therefore fails. 

\noindent In fact, for many permutations (of rotation type)  the constant roof function determines a class in the relative cohomology
$H^1(M, \Sigma, \R)$ which has a non-trivial component on the kernel of the forgetful map $H^1(M, \Sigma, \R) \to H^1(M, \R)$
(with respect to the Oseledets decomposition). Since the extension of the Kontsevich--Zorich cocycle (or the Rauzy--Veech cocycle) to the kernel of the forgetful map is {\it isometric}, the Veech criterion immediately implies the weak mixing property for almost all
IET's, since by the isometric behavior of the cocycle non-integer vector cannot converges to the integer lattice under the cocycle.

\smallskip
\noindent By this {\it isometric elimination} scheme, outlined above, Veech \cite{Ve84} established the weak mixing property for Lebsgue almost all IET's, for all permutations or rotation class, which are not rotations. Permutations of rotation class are those equivalent to rotations (cyclic) permutations under Rauzy operations. The particular case of IET's on $3$-intervals (which are not rotations) was known earlier from the work of Katok and Stepin \cite{KS67}.

\smallskip 
\noindent  The completion of the proof of weak mixing for typical IET's (see \cite{AvF07}, \cite{AFS23}) is based on a non-linear, probabilistic elimination scheme. 

\smallskip
\noindent The {\it hyperbolic elimination} scheme analyzes the dynamics in $H^1(M, \mathbb R)$ modulo $H^1(M, \mathbb Z)$ of {\it lines segments} approximately aligned with the top unstable Oseledets subspace contained in balls of given (sufficiently) small radius
around integer vectors in the cohomology vector space.  There are two main effects:

\begin{enumerate}
\item {\it Repulsion}:  Under the hypothesis that the second Lyapunov exponent of the Kontsevich--Zorich cocycle $\lambda_2>0$  line segment are repelled away from the nearest integer point (under the action of the Kontsevich--Zorich cocycle); 
\item {\it Spreading}:  For sufficiently {\it long excursions} of the Teichm\"uller orbit, a line segment can get dilated and spread to several line segments near many different integer points (after intersection with balls centered at those points). 
\end{enumerate}
The strategy of the argument is to prove that the probability that a given segment remains inside a fixed neighborhood of the integer
lattice converges to zero as time diverges:  {the probability of escaping  elimination goes to zero with time.} This conclusion follows
from an estimate of the probability that the repulsion mechanism pushes a segment out of the given neighborhood against the probability that the segment ``survives'' by jumping near another lattice point during an excursion of the Teichm\"uller obit. The latter
is small since excursions of Teichm\"uller geodesics outside of the ``thick part'' of the moduli space are rare  (see \cite{At06} or 
\cite{AGY06} although the original argument in \cite{AvF07} bypasses such stronger results on excursions). 

\noindent The key underlying dynamical feature that makes the probabilistic argument possible are the mixing properties of the Rauzy--Veech induction (rather of the induced maps on compact sub-simplices)  or of the Teichm\"uller flow, which are exponentially mixing dynamical systems (see \cite{AvF07},  \cite{AGY06}, \cite{AB07}).

\subsection{Effective weak mixing for translation flows}
\label{ssec:typical_eff_wm} 

\noindent The quantitative theory of weak mixing for translation flows and Interval Exchange Transformations is based on an effective
version of the Veech criterion, which, as the Veech criterion in Theorem~\ref{thm:criterion},  can be derived by methods of twisted Hodge theory.

\begin{theorem} 
\label{thm:effective_criterion}
(Effective Veech Criterion) Assume that for some compact set $K \subset  \mathcal H^{(1)}_\kappa$ of positive measure and for
some   neighborhood $U \subset H^1_\kappa(M, \T)$ of the zero section  (or, equivalently, for some neighborhood  of $H^1_\kappa(M, \Z)$ in $H^1_\kappa(M, \R))$,  for a given $(M, h) \in \mathcal H^{(1)}_\kappa$  and $\lambda \in \R\setminus \{0\}$, we have 
$$
f_{K,U} := \limsup_{t>0} \frac{ \text{\rm Leb}( \{ s\in [0, t] \vert g_t(M,h) \in K \text{ and }   g^{KZ}_t (M, h,   \lambda [\re (h)]) \in U \})} {   \text{\rm Leb}( \{ s\in [0, t] \vert g_t(M,h) \in K\}) }    <1\,.
$$
Then there exist constants  $\alpha>0$  and $C(\lambda)>0$ such that, for  all weakly differentiable $f \in H^1(M)$ and for all
$(p,T) \in M\times \R^+$ (such that $x$ has an infinite forward horizontal orbit)   we have
$$
\left\vert \int_0^T  e^{2\pi\imath \lambda t} f \circ \phi^X_t (p) dt \right\vert  \leq C(\lambda) \Vert  f \Vert_{H^1(M)}   T^{1- \alpha}\,.
$$
\end{theorem} 

\noindent The proof of the effective Veech criterion in Theorem~\ref{thm:effective_criterion} is completely analogous to  that of 
effective unique ergodicity in Theorem \ref{thm:eff_UE_selfsim} (self-similar case)  and Theorem \ref{thm:eff_UE_typ} (typical case), in fact it is a ``twisted'' version of it.  We outline this analogy and the resulting argument below. 

\begin{proof} [Outline of the Proof of Theorem~\ref{thm:effective_criterion}]   {\it First Step}. For a fixed $\lambda \in \R$ (including in the untwisted case $\lambda=0$), we write the twisted integrals in terms of $1$-dimensional currents given by the orbit segments.  We view orbit arcs  $\gamma^X_T(p)$ of the horizontal flow  $\phi^X_\R$ as $1$-dimensional currents defined as
$$
\gamma^X_T(p) (\alpha) := \int_0^T  e^{2\pi \imath \lambda t}   (\imath_X \alpha) \circ \phi^X_t (p) dt \,, \quad \text{for all smooth $1$-form } \alpha \in W^1(M,h)\,.
$$
Since $\imath_X \re(h) \equiv 1$, we have that 
$$
\gamma^X_{\lambda,T}(p) ( f \re h) =   \int_0^T  e^{2\pi \imath \lambda t}   f\circ \phi^X_t (p) dt  \,.
$$
We note that the current $\gamma^X_{\lambda,T}(p)$, as current of integration along a codimension $1$ submanifold belong to the
space $W^{-1}(M,h)$ by the Sobolev trace theorem.

\smallskip
\noindent {\it Second step}. We argue that such currents are at bounded distance from the space of currents which are closed with respect to the twisted differential $$d_{h, \lambda} =d + 2 \pi \imath \lambda \re (h) \wedge\,.$$  Indeed, we can compute the twisted differential of the current $\gamma^X_{\lambda,T}(p)$ as follows.  Let $f \in C^\infty(M)$, then we have
$$
\begin{aligned}
\vert \gamma^X_{\lambda,T}(p) (d_{h,\lambda}f) \vert & := \left\vert \int_0^T  e^{2\pi \imath \lambda t}   (\imath_X d_{h,\lambda}f) \circ \phi^X_t (p) dt \right\vert 
\\ &=  \left\vert   \int_0^T  e^{2\pi \imath \lambda t}   (Xf +2\pi \imath \lambda f  ) \circ \phi^X_t (p) dt  \right\vert  \\
&=  \left\vert   \int_0^T  \frac{d}{dt} \left( e^{2\pi \imath \lambda t}   f\circ \phi^X_t (p) \right) dt   \right\vert  \\ 
&=  \left\vert  e^{2\pi \imath \lambda T}   f\circ \phi^X_T(p)- f(p)  \right\vert  \leq  2 \Vert  f \Vert_{W^{2}(M,h)}  \,.
\end{aligned} 
$$
Thus,  $\Vert d_{h, \lambda} \gamma^X_{\lambda,T}(x) \Vert_{W^{-2}(M,h)}  \leq  2$, for all $(p,T)\in M\times \R^+$, and  since the operator
$d_{h,\lambda}$ is elliptic,  hence its restriction to the orthogonal of its kernel $\mathcal Z^{-1}_{\lambda}(M,h)$ has bounded inverse, it is possible to prove that there exists a constant $C(M) >0$ (uniformly bounded on compact subsets of the moduli space) such that 
$$
\inf_{z \in \mathcal Z^{-1}_{ \lambda}(M,h)}  \Vert   \gamma^X_{\lambda,T}(p)  -z \Vert_{W^{-1}(M,h)}  \leq  C(M)
$$
Since, as we have just proved, the current $ \gamma^X_{\lambda,T}(p)$ is at uniformly bounded distance from twisted-closed currents,
the argument is reduced to an estimate on the growth of the (Sobolev) norm of {\it twisted closed currents} under the restriction
of the ``transfer'' cocycle on the bundle $W^{-1}_\kappa (M)$ of currents of Sobolev order $-1$ to the sub-bundle 
$\mathcal Z^{-1}_{\kappa, \lambda}(M)$  of twisted closed currents.  

\smallskip
\noindent {\it Third step}.  The proof of bounds on the growth of the (Sobolev) norm  of order $-1$  of twisted  closed currents
can be reduced to the growth of the Hodge norm of twisted cohomology classes once we have prove that  there exists an invariant norm on the sub-bundle $\mathcal E^{-1}_{\kappa, \lambda}(M)$  of twisted exact currents.  The invariant norm is defined as
follows. By definition, a current $\gamma \in W^{-1} (M,h) $ is twisted exact if   $\gamma\in d_{h, \lambda}(L^2(M, \text{\rm area}_h)) $, hence we can define
$$
\Vert\!\vert  \gamma \vert\!\Vert_{h, \lambda}  :=  \Vert  U_{h, \lambda}  \Vert_{L^2(M, \text{\rm area}_h)}  
 \quad \text{ if }   \gamma = d_{h, \lambda} U_{h, \lambda} \in \mathcal E^{-1}_{\lambda}(M,h)   \,.
$$
{\it Fourth step} At this point, since we have reduced bounds on the currents $\gamma^X_{\lambda,T} (p)$ (in Sobolev norm)
to bounds on twisted cohomology classes (in Hodge norm), by Lemma \ref{lemma:spectral_gap_twist} we can derive the the following bound. We note that on compact sets of the moduli space, the twisted cohomology is finite dimensional,  the quotient
norm induced by the Sobolev norms are equivalent to the Hodge norm. 

\noindent Let $K\subset {\mathcal  H}^{(1)}_\kappa$ be any given compact set. For any $(M,h)$ there exists a constant $C_K (M,h)>1$ such that, for any $T= e^t$ such that $g_t(M,h) \in K$, for any $p\in M$ (with infinite horizontal forward orbit) we have 
\begin{equation}
\label{eq:twist_integral_bound} 
\Vert \gamma^X_{\lambda,T} (p) \Vert_{W^{-1}(M,h)}  \leq  C_K(M,h)  \exp \left (  \int_0^{\log T}  \Lambda^\# (g^{KZ}_s(M,h,\lambda[ \re(h)] ) ds \right)\,
\end{equation}
Let $K$ be a compact set of positive measure and $(M,h)$ be such that 
$$
\liminf_{t>0} \text{\rm Leb} \{ s\in [0,t ] \vert  g_s(M, h) \in K\}  = \mu_K >0\,.
$$
Since by Lemma~\ref{lemma:Lambda_twist}   there exists a constant $\Lambda^\#_{\rm max} <1$ such that 
$$
\Lambda ^\# (M, h, \eta)   \leq \Lambda^\#_{\rm max} <1 \,, \text { for all }  (M,h )\in K \text{ and }  (M, h , \eta) \not \in  U\,,
$$
it follows from the bound in formula~\eqref{eq:twist_integral_bound} that 
$$
\Vert \gamma^X_{\lambda,T} (p) \Vert_{W^{-1}(M,h)}  \leq C_K(M,h) \exp\left ( 1- \mu_K  (1-f_{K,U})  (1- \Lambda^\#_{\rm max}) ) \log T \right) \,,
$$
that is, for $\zeta:= \mu_K (1-f_{ U})(1- \Lambda^\#)   >0$, 
we have 
\begin{equation}
\label{eq:bound_best_returns}
\Vert \gamma^X_{\lambda,T} (p) \Vert_{W^{-1}(M,h)}  \leq C_K(M,h)  T^{1-\zeta} \,.
\end{equation} 
Finally, let $(T_k)= (e^{t_k})$ with $t_0=0$ and$(t_k)$  for $k\geq 1$ a sequence of return times of the forward orbit $g_{\R^+} (M,h)$  to the compact set $K \subset \mathcal H^{(1)}_\kappa$ such that 
$$
\lim_{k \to \infty}  \frac{t_k}{k}  =  \mu_K >0\,.
$$
{\it Fifth Step}. The argument is completed by decomposing  any horizontal orbit of length $T>0$ into orbits arcs of lengths which belong to the sequence  $(T_k)$ (to which the bound \eqref{eq:bound_best_returns} can be applied). In fact for any $T >1$ we have the (Ostrowski's) decomposition 
$$
T = \sum_{k=0}^N   m_k  T_k  
$$
with $T_N \leq T \leq T_{N+1}$ and $m_k T_k \leq T_{k+1}$, for all $k <N$, hence by the bound in 
formula~\eqref{eq:bound_best_returns} we derive
$$
\Vert \gamma^X_{\lambda,T} (p) \Vert_{W^{-1}(M,h)}  \leq C_K(M,h)  \sum_{k=1}^N  m_k T_k^{1-\zeta} \,.
$$
For every $\epsilon \in (0, \mu_K)$ there exists $k_\epsilon \in \N$ such that 
$$
(\mu_K-\epsilon) k  \leq t_k  \leq   (\mu_K+\epsilon) k \,, \quad \text{ for all } k\geq k_\epsilon\,,
$$
hence there exists a constant $C_{K, \epsilon} >1$ such that 
$$
m_k \leq \frac{T_{k+1} }{T_k} =e^{t_{k+1} -t_k}  \leq     C_{K, \epsilon}  e^{2 \epsilon k}\,.
$$
There exists therefore a constant $C_{K, \epsilon} (\zeta) >1$ such that
$$
\begin{aligned} 
 \sum_{k=1}^N  m_k T_k^{1-\zeta}  &\leq  C_{K, \epsilon}   \sum_{k=1}^N  e^{ [(1-\zeta)(\mu_K+ \epsilon)  +2\epsilon] k} 
\\ & \leq  C_{K, \epsilon}(\zeta)  e^{ [(1-\zeta)(\mu_K+ \epsilon)  +2\epsilon] N}
 \leq   C_{K, \epsilon} (\zeta) T^{ 1-\alpha} 
\end{aligned} 
$$
with 
$$
\alpha = 1- (1-\zeta)\frac{\mu_K+ \epsilon}{\mu_K-\epsilon}   +2\epsilon\,.
$$
Since there exists  $\epsilon>0$ such that in the above formula $\alpha >0$, the argument is complete
(in fact, the bound in the statement  holds for any  exponent $\alpha <\zeta$).
\end{proof} 

\noindent The proof of (Masur--Veech) typical {\it effective} weak mixing of translation flows is completed by a linear approximation argument analogous to that outline for proof of typical weak mixing.  Namely, motivated by the effective Veech criterion (Theorem~\ref{thm:effective_criterion}),  in analogy with Definition \ref{def:weak_stable}, we introduce the effective weak stable space:

\begin{definition}  
\label{def:eff_weak_stable}
The {\bf effective weak stable space}  $\mathcal W^{eff}_{(M,h)} \subset H^1(M, \R)$ at a translation surface 
$(M,h) \subset \mathcal H^{(1)}_\kappa$ is the set  defined as follows. 

\noindent Let $\mathcal K$ be a countable exhaustion of the
moduli space $\mathcal H^{(1)}_\kappa$ by compact subsets and let  $\mathcal U$ a countable basis of 
neighborhoods of  $H^1_\kappa(M, \Z) \subset H^1_\kappa(M, \R)$. We then define 
$$
\begin{aligned}
&\mathcal W^{eff}_{(M,h)}  :=  \cap_{K \in \mathcal K } \cap_{U \in \mathcal U}   \{   c \in H^1(M,\R) \vert  \\
&\quad \limsup_{t>0}  \frac{ \text{\rm Leb}( \{ s\in [0, t] \vert g_t(M,h) \in K \text{ and }   g^{KZ}_t (M, h,   \lambda [\re (h)]) \in  U  \})}
 {  \text{\rm Leb}( \{ s\in [0, t] \vert   g_t (M, h) \in  K\})  }    =1   \}\,.
\end{aligned} 
$$
\end{definition} 

\noindent In other terms, in the definition of the effective stable space {\it the convergence takes place  only in 
average} and excursion of the cohomology class outside of any given neighborhood of the integer lattice are possible provided 
their frequency converges to zero in the large time limit.

\noindent Despite this importance difference, all the basic properties of the stable space (in particular concerning its Hausdorff
dimension) hold for the effective weak stable space, and the linear elimination procedure outlined in section \ref{ssec:typical_wm}
can be carried out (see \cite{F22a}), as well as the non-linear, probabilistic elimination (see \cite{AFS23}).

\smallskip
\noindent By the effective Veech criterion (Theorem \ref{thm:effective_criterion} and by the  linear elimination argument we can then derive the following:
\begin{theorem} \cite{F22a} 
For each stratum $\mathcal H^{(1)}_\kappa$ of translation surfaces there exists a constant $\alpha_\kappa>0$
  such that, for the Masur--Veech typical  translation surface $(M,h) \in \mathcal H^{(1)}_\kappa$ and for all 
  $\lambda \in \R$ there exists a constant  $K_h(\lambda) >0$  such that  the horrizontal translation flow satisfies the following effective weak mixing estimate . For all functions 
  $f \in H^1(M)$  (the Sobolev space of functions with square integrable first weak derivative) of zero average, and for all 
  $(p,T) \in M \times\R^+$ such that $p$ has infinite forward orbit, we have
  $$
  \Big \vert   \int_0^T  e^{2\pi \imath \lambda t}   f \circ \phi^{X}_t (p)  dt   \Big\vert   \leq  K_h(\lambda) T^{1-\alpha_\kappa}\,.
  $$
\end{theorem} 
\noindent The above theorem implies the last equivalent property in Definition~\ref{def:wm_eff} for the space $W^1(M)$, hence it concludes (see also Exercise~\ref{ex:wm_eff}) the outline of the proof of polynomial weak mixing for typical translation flows stated in Theorem~\ref{thm:weak_mixing_eff}. For the proof of effective weak mixing for IET's, based on the effective Veech criterion and the probabilistic elimination procedure, we refer the reader to \cite{AFS23}.

\end{document}